\documentclass[pdflatex,sn-mathphys-num]{sn-jnl}  

\usepackage{graphicx}
\usepackage{multirow} 
\usepackage{amssymb,amsfonts}
\usepackage{amsthm} 
\usepackage{mathrsfs} 
\usepackage[title]{appendix} 
\usepackage{xcolor} 
\usepackage{textcomp} 
\usepackage{manyfoot} 
\usepackage{booktabs} 
\usepackage{algorithm} 
\usepackage{algorithmicx} 
\usepackage{algpseudocode} 
\usepackage{listings} 
\usepackage{subfigure}
\usepackage{framed}
\usepackage{adjustbox}
\usepackage{comment} 
\usepackage[tbtags]{amsmath}
\usepackage{pdflscape}
\usepackage{anyfontsize}
\usepackage{enumerate}
 \DeclareMathOperator{\sech}{sech}
 
\theoremstyle{thmstyleone} 
\newtheorem{theorem}{Theorem} 
\newtheorem{proposition}[theorem]{Proposition} 
\newtheorem{lemma}{Lemma} 

\theoremstyle{thmstyletwo} 
\newtheorem{example}{Example} 
 
\newtheorem{assumption}{Assumption}[section]
\theoremstyle{thmstylethree} 
\newtheorem{definition}{Definition}
\newtheorem{note}{Note}[section]
\newtheorem{cor}{Corollary}[section]

\newcommand{\STAB}[1]{\begin{tabular}{@{}c@{}}#1\end{tabular}}

\usepackage{array}

\makeatletter
\newcommand{\thickhline}{ 
    \noalign {\ifnum 0=`}\fi \hrule height 1.5pt
    \futurelet \reserved@a \@xhline
}
\newcolumntype{"}{@{\hskip\tabcolsep\vrule width 1pt\hskip\tabcolsep}}
\makeatother

\begin{document}

\title[Trust-Region Method for Set Optimization Problems]{Trust-Region Method for Optimization of Set-Valued Maps Given by Finitely Many  Functions}
 
\author[1]{\fnm{Suprova} \sur{Ghosh}}\email{suprovaghosh.rs.mat19@itbhu.ac.in}

\author*[1]{\fnm{Debdas} \sur{Ghosh}}\email{debdas.mat@iitbhu.ac.in}

\author[2]{\fnm{Christiane} \sur{Tammer}}\email{christiane.tammer@mathematik.uni-halle.de}

\author[3]{\fnm{Xiaopeng} \sur{Zhao}}\email{zhaoxiaopeng.2007@163.com}

\affil[1]{\orgdiv{Department of Mathematical Sciences}, \orgname{Indian Institute of Technology (BHU)}, \orgaddress{\city{Varanasi}, \postcode{221005}, \state{Uttar Pradesh}, \country{India}}}

\affil[2]{Department of Mathematics, Martin-Luther-University Halle-Wittenberg, 06099, Halle (Saale), Germany}

\affil[3]{School of Mathematical Sciences, Tiangong University, Tianjin, 300387, China}

\abstract{
In this article, we develop a trust-region technique to find critical points of unconstrained set optimization problems with the objective set-valued map defined by finitely many twice continuously differentiable functions. The technique is globally convergent and has the descent property. To ensure the descent property, a new rule of trust-region reduction ratio is introduced for the considered set-valued maps. In the derived method, to find the sequence of iteration points, we need to perform one iteration of a different vector optimization problem at each iteration. Thus, the derived technique is found to be not a straight extension of that for vector optimization. The effectiveness of the proposed algorithm is reported through performance profiles of the proposed approach with the existing methods on various test examples. A list of test problems for set optimization is also provided.}



\maketitle

\section{Introduction}\label{sec1}

In this article, we consider the set optimization problem
\begin{equation*} 
    \underset{{x \in \mathbb{R}^n}}{\text{minimize}} ~~  F(x), 
\end{equation*}
where the objective map $F:\mathbb R^{n} \rightrightarrows \mathbb R^{m}$ is given by 
\begin{align}\label{particl_set} 
F(x) := \left\{ f^{1}(x), f^{2}(x), \ldots, f^{p}(x) \right\}, ~x \in \mathbb R^{n},
\end{align}
where $f^{1}, f^{2}, \ldots, f^{p}: \mathbb R^{n}\to \mathbb R^{m} $ are twice continuously differentiable functions. There are two main approaches to defining the solution concept of a set optimization problem: vector approach and set approach. \emph{Vector approach} compares elements belonging to the graph of ${F}$: 
 an element $x_0 \in \mathbb{R}^n$ is called a minimal solution if there exists $y_0 \in F(x_0)$ such that $y_0$ is a minimal element of the collection of $y$'s that lie in $F(\mathbb{R}^n)$. 
\emph{Set approach}, on the other hand, considers a pre-order on the power set of the image space of $F$ and then identifies the minimal element among the set values of the objective map.

Existing algorithms in the literature for solving set optimization problems are primarily based on these two approaches, which can be broadly reclassified into six groups: 
scalarization-based methods \cite{steepmethset, Ehrgott, Eichfelder22, Gabriele, Ide, Ide214}, derivative-free methods \cite{Jahnderivativefree,Jahn2018, kobis}, sorting-based methods \cite{bGunther2019,Jahn}, branch and bound method  \cite{eichfelder2020algorithmic}, gradient-based descent methods  \cite{steepmethset,kumar2024nonlinear} and complete lattice approach \cite{lohne2025solution}.

These approaches have well-documented limitations.
Derivative-free methods suffer from many drawbacks, as detailed in \cite{steepmethset}. For the type of problems \eqref{particl_set} considered in this paper, sorting-based and branch-and-bound methods are simply not applicable because sorting-based methods require the feasible set to be finite, and branch-and-bound methods require the problem to be constrained in order to define the box that contains the feasible set. The steepest descent method is known to be slow, while the lattice approach is intractable in high dimensions. Given these limitations, it is natural to seek an alternative method that can handle nonconvexity, avoid explicit line searches, and still guarantee convergence.

\subsection{Related work}
One promising direction is the trust-region methodology, which has been shown to be both globally convergent and robust in various optimization settings.
Trust-region method for scalar optimization has been studied extensively in \cite{heinkenschloss1998trust,conn2000trust,absil2009accelerated}. This method has been applied to solve multi-objective problems in \cite{Qu2013} and \cite{villacorta2014trust}. Jauny et al. \cite{ghosh2022trust} developed a trust-region interior-point technique to generate the Pareto optimal solution for multi-objective optimization problems.
Ryu and Kim \cite{ryu2014derivative} proposed a derivative-free trust-region algorithm for black-box biobjective optimization problems. Thomann and Eichfelder \cite{thomann2019trust} presented a trust-region method for heterogeneous multi-objective optimization problems. Carrizo et al. \cite{Trrionglion} adapted the notions of descending condition and predicted reduction to the multi-objective case. Very recently, Mohammadi and Cust{\'o}dio \cite{mohammadi2024trust} designed a trust-region algorithm with the help of the concepts of extreme point step and localization step that helps to estimate an approximation of the complete Pareto front in a multi-objective problem. 
While these studies demonstrate the flexibility of trust-region methods, they have so far been developed only for scalar or multi-objective problems, and not for vector optimization in its general form, nor for set optimization problems. This gap motivates the present work.

\subsection{Motivation}
 To overcome the limitations identified above, we aim to derive a trust-region method for the class of set optimization problems considered in this paper. This method not only exhibits global convergence property but also does not use any explicit line search. Especially for nonconvex optimization, trust-region methods \cite{conn2000trust} have been found to be very effective for scalar and vector optimization problems. 
Thus, in this study, we present a set approach for set optimization problems by employing a trust-region-based algorithm with indirect scalarization. We finally solve the resulting scalar optimization problem by using the proposed trust-region algorithm for set optimization.
Although trust-region methods have been studied for multi-objective optimization, it has not yet been studied for vector optimization and set optimization.

\subsection{Contributions}

In this article, we derive a trust-region method for set optimization problems; when the set-valued objective map is a singleton-vector-valued map, the proposed method acts as a trust-region method for vector optimization. In order to update the trust-region radius and to determine acceptance of a step, we introduce a new rule of reduction ratio for the setting of a set-valued objective map. Furthermore, we show that, similar to the conventional trust-region methods, the global convergence property holds for set optimization problems. Although the proposed trust-region method is motivated from the works of Bouza et al. \cite{bouza2019unified}, Fliege et al. \cite{NewtonMethod}, Carrizo et al. \cite{Trrionglion}, however, it is distinct and novel in the following manner.
\begin{itemize}
\item  Bouza et al. \cite{bouza2019unified} proposed a steepest descent method with first-order convergence using the Gerstewitz scalarization. Our method ensures second-order global convergence using the oriented distance scalarization function. Furthermore, in our method, the selection of objectives at each iteration, based on the chosen element from the partition, is different from that in \cite{bouza2019unified} as there is no line search used in the proposed method.

\item  Fliege et al. \cite{NewtonMethod} addressed multi-objective optimization. In contrast, we tackle optimization problems involving set-valued objective maps, where the solution concept is given by the lower set-less relation, which generalizes the usual dominance in vector optimization. At each iteration, our algorithm creates a new vector subproblem depending on the partition set  $P_{x_k}$ (which commonly changes at every iteration) and thus cannot be treated as a particular type of a vector optimization problem, as in \cite{NewtonMethod}.

\item Our method differs from \cite{Trrionglion} in the same way as it does from \cite{NewtonMethod}. Further, since our approach uses a general ordering cone rather than the canonical ordering cone as in Carrizo et al. \cite{Trrionglion}, the approach in \cite{Trrionglion} cannot be directly generalized to the proposed one. In fact, we need to introduce a reduction ratio based on the oriented distance function for a general ordering cone, tailored to set-valued maps, instead of the conventional reduction ratio used in \cite{Trrionglion}.
\end{itemize}

\subsection{Organization}
The proposed study is organized as follows: Section \ref{section1} presents the basic terminologies and recalls some pre-existing definitions. It also introduces the notion of critical points and descent directions for set-valued maps. In Section \ref{Section3}, we describe the proposed trust-region-based approach in detail. We also provide a step-by-step algorithm and report its well-definedness. Section \ref{global_conv_ana} is devoted to the global convergence analysis of the method. Section \ref{Numerical_Experiment} presents numerical experiments and results by running the proposed Algorithm \ref{avds1} on a few test cases. Finally, Section \ref{sect6} draws a few potential future scopes.

\section{Preliminaries and Terminologies}
\label{section1} 
\noindent
In this section, we provide notations and definitions that are used throughout the paper. The notations $\mathbb R$, $\mathbb R_{+}$ and $\mathbb R_{++}$ stand for the set of real numbers, the set of nonnegative real numbers, and the set of positive real numbers, respectively; also, $\mathbb R^{m}_{+}$ and $\mathbb R^{m}_{++}$ denote the nonnegative and positive hyperoctant of $\mathbb R^{m}$, respectively. 

The class of all nonempty subsets of $\mathbb R^{m}$ is denoted by  $\mathscr{P}(\mathbb R^{m})$. Further, for an $\mathcal{A} \in \mathscr{P}(\mathbb R^{m})$, its interior and boundary are denoted by $\operatorname*{int}({\mathcal A})$ and $\operatorname*{bd}({\mathcal A})$, respectively. 
The notation $\lVert \cdot \rVert$ denotes a norm. The cardinality of a finite set $\mathcal A$ is denoted by $\lvert \mathcal{A} \rvert$. For any $k \in \mathbb N$, we denote $[k] := \{1, 2, \ldots, k\}$.

A set $K \in  \mathscr{P}(\mathbb R^{m})$ is called a cone in $\mathbb{R}^m$ if  $y \in K$ implies $ty\in K $ for all $t \ge 0$. A cone $K$  is convex  if $K+K = K$, pointed  if $ K \cap (-K) = \{0\}$, and solid if $\operatorname*{int}(K) \neq \emptyset$. If $K$ is a convex  and pointed cone in $\mathbb{R}^m$, it defines a partial ordering on $\mathbb R^{m}$ given by
\begin{align*}
    y \preceq_{K}z \Longleftrightarrow z -y \in K,
\end{align*}
and moreover, if $K$ is solid, it defines a strict ordering on $\mathbb R^{m}$ given by
\begin{align*}
     y \prec_{K} z \Longleftrightarrow z -y \in \operatorname*{int}(K).
\end{align*}

Throughout the rest of the paper, wherever we use the notation $K$, it is a closed, pointed, convex, and solid cone in $\mathbb{R}^m$. 
\begin{definition} {\cite{khran}}  
Let $ {\mathcal A} \in \mathscr P(\mathbb R^{m})$ and $y_{0}\in {\mathcal A}$. 
\begin{enumerate}[(i)]
    \item An element $y_{0}$ is called a $K$-minimal element of ${\mathcal A}$ if $(y_{0}-K)\cap {{\mathcal{A}}}= \{ y_{0}\}$. The collection of all $K$-minimal elements of $\mathcal A$ is denoted by $\operatorname*{Min}(\mathcal{A}, K )$.
    
    \item An element $y_{0}$ is called a weakly $K$-minimal element of ${\mathcal A}$ if $(y_{0}-\operatorname*{int}{(K)})\cap {{\mathcal{A}}}= \emptyset$. The collection of all weakly $K$-minimal elements of $\mathcal{A}$ is denoted by $\operatorname*{WMin}(\mathcal{A}, K)$.
\end{enumerate}
\end{definition}

We next recall the lower set-less relation $\preceq^{l}_K$ on $\mathscr P(\mathbb R^{m})$ with respect to a given $K$ (see \cite{kuoiwa}). Let ${\mathcal A}, {\mathcal B} \in \mathscr P(\mathbb R^{m})$. The relation $\preceq^{l}_K$ on $\mathscr{P}(\mathbb R^{m})$ is defined by 
  \begin{align*}
    {\mathcal A} \preceq^{l}_{K} {\mathcal B}  ~\Longleftrightarrow~ {\mathcal B}\subseteq {\mathcal A} +  K.
\end{align*}
If $K$ is solid, the strict lower set-less relation $\prec^{l}_K$ on $\mathscr P(\mathbb R^{m})$ is given by 
\begin{align*}
    {\mathcal A} \prec^{l}_{K} {\mathcal B}  \Longleftrightarrow 
    {\mathcal B}\subseteq {\mathcal A} +  \operatorname*{int}(K).
\end{align*}

 Throughout the paper, the $j$-th component function of the function $f^i: \mathbb R^{n}\to \mathbb R^{m}$ is denoted by $f^{i, j}$, i.e., 
\[f^{i}(x) := \left(f^{i, 1}(x), f^{i, 2}(x), \ldots, f^{i, m}(x)\right)^\top, ~ x \in \mathbb{R}^n, ~i \in [p].\]
Often we present the map $F$ in \eqref{particl_set} simply by the shorthand $F:=\{f^i\}_{i \in [p]}$. 
Associated to a given set-valued map $F: \mathbb R^{n} \rightrightarrows \mathbb R^{m}$, we study  the unconstrained set optimization problem 
\begin{equation}\label{fgcx}
(\preceq^{l}_{K})\text{--}\min_{x\in \mathbb R^{n}}  F(x).
\tag{\text{$\mathcal{SOP}^{l}_K$}}
\end{equation}
We say that $\bar x\in \mathbb R^{n}$ is a (resp., weakly) $\preceq^{l}_{K}$-minimal solution of (\ref{fgcx}) if there does not exist any $ x\in \mathbb R^{n}$ such that (resp., $F(x) \prec^{l}_K F(\bar x)$) $F(x) \preceq^{l}_K F(\bar x)$.  Here, we note that if we consider $p = 1$, then $F(x) = {f^{1}(x)}$, and the (resp., weakly) $\preceq^{l}_{K}$-minimal solution of (\ref{fgcx}) reduces to the (resp., weakly) $K$-minimal solution of the vector optimization problem $f^{1}(x)$.

\begin{definition}\cite{steepmethset} \label{active_index} 
\begin{enumerate}[(i)]
\item For a given vector $ v \in \mathbb R^{m}$, the map $I_{v}: \mathbb R^{n} \rightrightarrows [p]$ is defined by
\begin{align*}
    I_{v}(x)  := \{ i \in I(x):  f^{i}(x)= v\}, 
\end{align*}
where the map $I: \mathbb R^{n} \rightrightarrows [p] $ is given by 
$I(x) := \{ i \in [p]: f^{i}(x) \in \operatorname*{Min}({F}(x),  K)\}$.  
\item The cardinality function $ \omega : \mathbb R^{n} \to \mathbb R$ is defined by $\omega (x)  := \lvert  \operatorname*{Min} (F(x),  K)\rvert.$ 
\end{enumerate}
\end{definition}

\begin{definition} \cite{steepmethset} For a given $\bar x \in \mathbb R^{n}$, consider an enumeration $\{v^{\bar x}_{1}, v^{\bar x}_{2}, \ldots v^{\bar x}_{\omega(\bar x)} \}$ of  the set $\operatorname*{WMin}(F(\bar x),  K) $.
  The \emph{partition  set} at $\bar x$ is defined by 
  \begin{align*}
      {P}_{\bar x} := I_{v^{\bar x}_{1}}(\bar x) \times I_{v^{\bar x}_{2}}(\bar x) \times \cdots \times I_{v^{\bar x}_{\omega(\bar x)}}(\bar x).
  \end{align*}
\end{definition}
Throughout the paper, for a given $\bar x \in \mathbb{R}^n$, we denote $\bar \omega := \omega(\bar x)$, and a generic element of the partition set $P_{\bar x}$ is denoted by $\bar a := (\bar a_1, \bar a_2, \ldots, \bar a_{\bar \omega})$. 

\begin{definition}\label{ghtr_ihn}  \cite{steepmethset}  
A point $\bar x$ is said to be a regular point of the set-valued map $F: \mathbb{R}^n \rightrightarrows \mathbb{R}^n$ if 
$\operatorname*{Min}({F}(\bar x), K) = \operatorname*{WMin}
 ({F}(\bar x), K)$, and 
the cardinality function $\omega$ as given in Definition \ref{active_index} is constant  in a neighbourhood of $\bar{x}$.
\end{definition}

\begin{lemma}\label{regularity_condition}
\emph{\cite{steepmethset}} If $\bar x$ is a regular point of $F$, then there exists a neighborhood $\mathcal{N}$ of $\bar x$ such that 
$\omega(x) = \omega(\bar x) \text{ and } P_x \subseteq P_{\bar x} \text{ for all } x \in \mathcal{N}. $
\end{lemma}

\begin{definition}[Critical point for (\ref{fgcx})\label{spcritic}]
A point $ \bar x \in \mathbb R^{n}$ is said to be a $K$-critical point of \eqref{fgcx} if there does not exist any $\bar a \in P_{\bar x}$ and $\bar s \in \mathbb{R}^n$ such that 
$$ \nabla f^{\bar a_j}(\bar x)^{\top} \bar s \prec_K 0 \text{ for all } j \in [\omega(\bar x)], $$ 
where for any $j \in [\omega(\bar x)]$, the notation $\nabla f^{\tilde{a}_{j}}(\bar x)^{\top} s \in \mathbb{R}^m$ is given by $$\nabla f^{\tilde{a}_{j}}(\bar x)^{\top} s :=  
\left(   
 \nabla f^{\tilde{a}_{j},1}(\bar x)^{\top}{s}, 
 \nabla f^{\tilde{a_{j}},2}(\bar x)^{\top}{s},  
 \ldots, 
 \nabla f^{\tilde{a_{j}},m}(\bar x)^{\top}{s}
\right)^\top, s \in \mathbb{R}^n. $$  

Note that if we consider $p = 1$, then $F(x) = \{f^{1}(x)\}$ and then $K$-critical point of \eqref{fgcx} reduces to the $K$-critical point of the vector optimization problem $f^{1}(x)$.
\end{definition}

It is trivial to see that a weakly $K$-minimal point of \eqref{fgcx} is a $K$-critical point of \eqref{fgcx}.

\begin{definition}[$\preceq^{l}_{K}$-descent direction for set-valued maps]\label{descent_setopt}
A vector $s \in \mathbb R^{n}$ is said to be  a $\preceq^{l}_{K}$-descent direction of the set-valued map $F := \{ f^{i}\}_{i \in [p]}$ at $\bar x$ if there exists $t_{0} > 0$ such that $
\{f^{i}(\bar x + t s) \}_{i \in [p]} \prec^{l}_{K} \{ f^{i}(\bar x)\}_{i \in [p]} ~\text{for all}~ t \in (0, t_{0}].$

Note that if we consider $p = 1$, then $F(x) = \{f^{1}(x)\}$ and then $\prec^{l}_{K}$-descent direction of \eqref{fgcx} reduces to the ${K}$-descent direction of the vector optimization problem $f^{1}(x)$.
\end{definition}

\begin{lemma}\label{rtyrrsv}Let $\bar x \in \mathbb{R}^m$, $\bar{\omega}:= \omega (\bar x)$, and $\widetilde{K} \in \mathscr{P}(\mathbb R^{m\bar{\omega}})$ be the cone $\Pi^{\bar{\omega}}_{j=1} {K}.$  For all $a := (a_{1}, a_{2}, \ldots, a_{\bar{\omega}}) \in {P}_{\bar x}$, define the function ${\widetilde{f}}^{a}: \mathbb R^{n} \to \Pi_{j =1}^{\bar{\omega}} \mathbb R^{m}$ by  
\begin{align}\label{vector_valued}
{\widetilde{f}}^{a}(x) :=  \left( 
 f^{a_{1}}(x), f^{a_{2}}(x), \ldots, f^{a_{\bar{\omega}}}(x)\right)^\top. 
\end{align}
Then, $\bar{x}$ is a $K$-critical point of \eqref{fgcx} if and only if $\bar x$ is a $\widetilde{K}$-critical point of the following vector optimization problem for every $a \in {P}_{\bar x}$: 
\begin{equation*}\label{vec_pr}
(\preceq_{\widetilde{ K}})\text{--} \min_{x \in \mathbb R^{n}}  {\widetilde{f}}^{a}(x). 
\tag{\text{$\mathcal{VOP}_{a}$}}
  \end{equation*}
  \end{lemma} 
\begin{proof}
Similar to \cite[Lemma 3.1]{bouza2019unified}.  
\end{proof}

\begin{definition}\label{oriented_distance}
\cite{Ansari2018} Let $\mathcal{A}\in \mathscr{P}(\mathbb R^{m})$. The function $\Delta_{\mathcal{A}}: \mathbb R^{m} \to \mathbb R  \cup \{\pm \infty\}$ defined by $
    \Delta_{\mathcal A}(y): = d_{\mathcal A}(y) -d_{\mathcal A^{c}}(y)$,
is called the oriented distance function, where $d_{\mathcal A}(y):=\inf\limits_{a \in \mathcal{A}}\lVert y-a \rVert$. 
\end{definition}
  
\begin{lemma} \emph{\cite{Ansari2018, zaffaroni}} \label{proori}
\begin{enumerate}[(i)]
     \item\label{aeruin} $\Delta_{-K}$ is Lipschitz continuous with Lipschitz constant $1$. 
     \item\label{fyf} $\Delta_{-K}(y)< 0$ if and only if $ y \in \operatorname*{int}(-K)$. 
     \item\label{hryx} $\Delta_{-K}(y)=0$ if and only if $ y \in \operatorname*{bd}(-K) $. 
     \item\label{ufbmvfd} $\Delta_{-K}(y)> 0$ if and only if $y \in \operatorname*{int}[(-K)^{c}] $. 
     \item \label{ubrsyu} $\Delta_{-K}(\lambda y)=\lambda \Delta_{\lambda^{-1}(-K)} (y)$ for all $ y \in \mathbb{R}^{m}$ and $ \lambda >0$. 
     \item \label{tewyueui}  $ \Delta_{-K}(y_{1}+y_{2})\le \Delta_{-K}(y_{1})+\Delta_{-K}(y_{2})$ and $\Delta_{-K}(y_{1})-\Delta_{-K}(y_{2})\le \Delta_{-K}(y_{1}-y_{2})$ for all $y_{1},y_{2}\in \mathbb R^{m}$.
     \item \label{htytuttd} Given $ y_{1}, y_{2} \in \mathbb R^{m}$, if  $y_{1} \prec_K y_{2}$ (resp., $y_{1} \preceq_K y_{2}$), then $\Delta_{-K}(y_{1}) < \Delta_{-K}(y_{2})$ (resp., $\Delta_{-K}(y_{1}) \le \Delta_{-K}(y_{2})$).
\end{enumerate}
\end{lemma}

\section{Trust-region method for set optimization}\label{Section3}

According to Lemma \ref{rtyrrsv}, any $K$-critical point of \eqref{fgcx} is a $\widetilde{K}$-critical point of \eqref{vec_pr} and vice-versa. Therefore, we may get tempted to think that the set optimization problem \eqref{fgcx} with the objective map given by finitely many vector-valued functions is just a vector optimization problem, and thus, all the existing methods for vector optimization problems can be directly applied to solve \eqref{fgcx}. However, this is not true. The reason is as follows. Note that the formulation of \eqref{vec_pr} depends on the partition set $P_{\bar x}$ at the point $\bar x$. If we aim to reach a $K$-critical point $x^\star$ of \eqref{fgcx} through an iterative sequence $\{x_k\}$ that converges to $x^\star$, then note that at each iterative point $x_k$, the formulation of \eqref{vec_pr} gets changed because the partition set $P_{x_k}$ gets changed across the iterates. Neatly, the vector optimization problem analogous to \eqref{vec_pr} at $x_k$ is the following problem:
\begin{equation*}\label{vop_at_x_k}
(\preceq_{{K^{\omega_k}}}) \text{--} \min_{x \in \mathbb R^{n}}  {\widetilde{f}}^{a^k}(x),~ a^{k} \in P_{x_k},
\tag{\text{$\mathcal{VOP}_{a^k}(x_k)$}}   
\end{equation*}
and commonly, the problem \eqref{vop_at_x_k} at $x_k$ is different than that at $x_{k + 1}$. Therefore, solving \eqref{fgcx} does not essentially amount to solving just one vector optimization problem \eqref{vec_pr} for all $a \in P_{x^\star}$.

On the other hand, from the observation that `if $\bar x$ is not a $\widetilde{K}$-critical point of \eqref{vec_pr} for at least one $a \in P_{\bar x}$, then $\bar x$ is not a $K$-critical point of \eqref{fgcx},' one may perceive that the computational effort in computing $P_{\bar x}$ and then solving \eqref{vec_pr} at a not $\widetilde{K}$-critical point is of no use to capture $K$-critical points of \eqref{fgcx}. This is also not true. In the following, we explore how the formulation of \eqref{vec_pr} can be exploited to identify \emph{a sequence of vector optimization problems} \eqref{vop_at_x_k} at noncritical points $x_k$'s, solving which we can arrive at a $K$-critical point of \eqref{fgcx}.

We start the process from any arbitrarily chosen initial point $x_0$. Through the trust-region iterative scheme that we propose below, we aim to generate a sequence $\{x_k\}$ that converges to a $K$-critical point of \eqref{fgcx}. Below, we describe how we proceed from  $x_{k}$ to $x_{k + 1}$. For this, we derive a direction of movement $s_{k}$ and then take the next iterate $x_{k+1}$ as $x_k + s_k$. The process of searching a direction $s_{k}$ at the $k$-th iterate is the following. At first, at the current iterate $x_{k}$, we identify the partition set $P_{x_k}$. Then, we choose \emph{tactfully} an $a^k$ from $P_{x_k}$ and formulate the corresponding \eqref{vop_at_x_k}; a $K^{\omega_k}$-descent direction $s_k$ of \eqref{vop_at_x_k} happens to be a 
$\preceq^{l}_{K}$-descent direction of (\ref{fgcx}) at $x_k$ (thanks to Proposition \ref{vop_descent_implies_descent_sop}). The process of choosing such an $a^k$ from $P_{x_k}$ is described below, which is based on identifying a necessary optimality condition of a weakly $\preceq^{l}_{K}$-minimal element of (\ref{fgcx}).  Throughout the rest of the paper, we denote $P_k := P_{x_k}$ and $\omega_k := \omega(x_k) = |P_k|$.

 \subsection{Choice of `\texorpdfstring{$a^k$}{Lg}' from \texorpdfstring{$P_{k}$}{Lg}}\label{subsect_31}
Before choosing an $a$ from $P_{{k}}$, we present a necessary condition for weak $\preceq^{l}_{K}$-minimal solution of 
(\ref{fgcx}). As a weak $\preceq^{l}_{K}$-minimal solution of (\ref{fgcx}) is a $K$-critical point of (\ref{fgcx}), from Definition \ref{spcritic}, we see that
\begin{align*}
    & ~\text{a point}~ x_{k}~ \text{is a $K$-critical point of} ~(\ref{fgcx})~\\
\Longleftrightarrow  & ~\text{for any} ~a \in P_{k} ~\text{and}~ s \in \mathbb R^{n},~\exists j_{0} \in [\omega_{k}]~ \text{such that}~\Delta_{-K}(\nabla f^{a_{j_{0}}}(x_{k})^{\top}s) \ge 0.  
\end{align*}
So, at a $\preceq^{l}_{K}$-critical point $x_{k}$, for any $a \in P_{k}$ and $s \in \mathbb R^{n}$, there exists $j_{0} \in [\omega_{k}]$ such that 
\begin{align}
    ~&~ \max_{j \in [\omega_{k}]} \left\{\Delta_{-K}(\nabla f^{a_{j}}(x_{k})^{\top}s + \tfrac{1}{2}s^{\top}\nabla^{2} f^{a_{j}}(x_{k})s), \Delta_{-K}(\nabla f^{a_{j}}(x_{k})^{\top}s) \right\} 
  \label{tdhfefurw}\\
  \ge ~&~ \Delta_{-K}(\nabla f^{a_{j_{0}}}(x_{k})^{\top}s) \ge 0 \nonumber.
\end{align} 
Let $\Omega_{\max}$ be the maximum allowed trust-region step-size across the iterates $x_{k}$'s. Denote $\mathcal{B} : = \{ s \in \mathbb R^{n}: \lVert s \rVert \le \Omega_{\max}\}$.
Thus, for any given $ x \in \mathbb R^{n}$, if we introduce a function  $\Theta_{x}: P_{x} \times \mathcal{B} \rightarrow \mathbb R$ by 
\begin{align}\label{namfun}
 &  \Theta_{x}(a,s) = \max_{j \in [\omega(x)]} \left\{ \Delta_{-K} \left(\nabla f^{a_{j}}(x)^{\top}s+\tfrac{1}{2}s^{\top}\nabla^{2} f^{a_{j}}(x)s\right), \Delta_{-K}(\nabla f^{a_{j}}(x)^{\top}s) \right\},  
\end{align}
where $a \in P_{x}$ and $s \in \mathcal{B}$, then by (\ref{tdhfefurw}), at a $K$-critical point $x_{k}$ of (\ref{fgcx}), we have 
 \begin{align}
     &\Theta_{x_{k}}(a, s) \ge 0  ~\forall~ a \in P_{k} ~\text{and}~ s \in \mathcal{B} \nonumber \\ 
  \implies & \forall a \in P_{k} : \inf_{s \in \mathcal{B}} \Theta_{x_{k}}(a, s) = 0 =  \Theta_{x_{k}} (a, 0) \label{yweqyc}.
 \end{align}
As $\mathcal{B}$ is compact, and for any $x \in \mathbb R^{n}$, the partition set $P_{x}$ is  finite, $\Theta_{x}$ attains its minimum over the set $P_{x} \times \mathcal{B}$. Let us define a function $\theta: \mathbb R^{n} \to \mathbb R$ by 
\begin{align}\label{rbdwbyew}
     \theta(x) := \min_{(a, s) \in P_{x} \times \mathcal {B}} \Theta_{x}(a,s).
 \end{align}
In view of (\ref{yweqyc}), at a $K$-critical point $x_{k}$ of (\ref{fgcx}), if for $({a}^{k}, {s}_{k}) \in P_{k} \times \mathcal{B}_{k}$ we have $\theta({x_k}) = \Theta_{x_{k}}( {a^{k}},{s_{k}}) $, then $ \theta(x_k)=0$, 
where $\mathcal{B}_k := \{x\in \mathbb{R}^n: \|x - x_k\| \le \Omega_k\}$ and $\Omega_{k}>0$. Hence, we obtain the following result. 

\begin{proposition}[Necessary condition for weakly $\preceq^{l}_{K}$-minimal points]\label{nrwelyd} If ${x_k}$ is either a weak $\preceq^{l}_{K}$-minimal point or a $K$-critical point of (\ref{fgcx}), and ${a}^{k} \in P_{k}$ and $ s_{k} \in \mathcal{B}$ be such that $\theta({x_k}) = \Theta_{x_{k}}( {a}^{k}, {s}_{k}) $, where $\Theta_{x}$ and $\theta$ are as defined in (\ref{namfun}) and (\ref{rbdwbyew}), respectively, then $\theta(x_{k}) =0.$
\end{proposition}

As a consequence of Proposition \ref{nrwelyd} and discussion in this  Subsection \ref{subsect_31}, 
if we can select an ${a}^{k} := \left({a}^{k}_1, {a}^{k}_2, \ldots,  {a}^{k}_{\omega_{k}}\right)$ from the partition $ {P}_{{k}}$ of the current iterate $x_{k}$ such that $\theta({x_k}) = \Theta_{x_{k}}( {a^{k}},{s_{k}})$, i.e.,
\begin{align}\label{dg_aux_17_01_2}
   &({a}^{k},  s_{k}) \in \underset{(a, s)\in P_{k}\times \mathcal B_{k}}{\operatorname{argmin}}  \Theta_{x_k}(a, s), 
\end{align}
then $\theta(x_{k}) \neq 0$ implies that $x_{k}$ is not a $K$-critical point of (\ref{fgcx}).  
Next, we employ a trust-region scheme to solve (\ref{vop_at_x_k}) for the particular ${a}^{k}\in P_{k}$ which satisfies \eqref{dg_aux_17_01_2}. Thereby, we get a trust-region step $s_k$ to progress to the next step $x_{k+1}:=x_k + s_k$. 

\subsection{Choice of the trust-region step `\texorpdfstring{$s_k$}{Lg}'}

Once an $a^k$ is chosen that satisfies \eqref{dg_aux_17_01_2}, we attempt to find a descent direction, if it exists, at $x_k$ of the objective function of (\ref{vop_at_x_k}) in $\mathcal{B}_k$.

We aim to choose an $s_k$ from $\mathcal{B}_k$ such that $(a^k, s_k)$ satisfies \eqref{dg_aux_17_01_2}. If with this $(a^k, s_k)$ we get $\Theta_{x_k}(a^k, s_k) = 0$, then we have reached a point $x_k$ at which a necessary condition for $K$-critical points of \eqref{fgcx} is satisfied, and we stop the process. If, however, $\Theta_{x_k}(a^k, s_k) \neq 0$, then $x_k$ is not a $K$-critical point of \eqref{fgcx}, and in this case, we revise the trust-region radius and the current iterate as described in the next subsections. 

For a point $x_k$ that is not $K$-critical of  \eqref{fgcx}, the model function of the objective function ${\widetilde{f}}^{{a}^{k}} := (f^{{a}^{k}_{1}}, f^{{a}^{k}_{2}}, \ldots, f^{{a}^{k}_{\omega_k}})^\top$ of \eqref{vop_at_x_k} inside $\mathcal{B}_k$ is taken as the following quadratic approximation ${\widetilde{m}}^{{a}^{k}}: = (m^{{a}^{k}_{1}}, m^{{a}^{k}_{2}}, \ldots, m^{{a}^{k}_{\omega_k}})^\top$: 
\begin{align}\label{mode_tru}
m^{{a}^{k}_{j}}(s):= \nabla f^{{a}^{k}_{j}}(x_{k})^{\top}s+ \tfrac{1}{2}s^{\top}\nabla^{2}f^{{a}^{k}_{j}}(x_{k})s,~ s \in \mathcal{B}_k,~ \text{ for each } j \in [\omega_k],  
\end{align}
\begin{align*} 
\text{where } & \nabla f^{{a}^{k}_{j}}(x_{k})^{\top}s := \left( 
\nabla f^{{a}^{k}_{j},1}( x_{k})^{\top}s, 
\nabla f^{{a}^{k}_{j},2}( x_{k})^{\top}s,  
\ldots, 
\nabla f^{{a}^{k}_{j},m}( x_{k})^{\top}s \right)^\top \text{ and }  \\ 
& s^{\top}\nabla^{2} f^{{a}^{k}_{j}}(x_{k})s := \left(   s^{\top}\nabla^{2} f^{{a}^{k}_j,1}(x_{k})s, 
s^{\top}\nabla^{2} f^{{a}^{k}_j,1}(x_{k})s, 
 \ldots, 
 s^{\top} \nabla^{2} f^{{a}^{k}_j,m}(x_{k})s \right)^\top.
\end{align*}
To generate a search direction $s_{k}$ in $\mathcal{B}_k$,  we solve the following problem: 
\begin{equation}\label{model_pro}
\min\limits_{s \in\mathcal{B}_k}\max\limits_{j \in [ \omega_{k}]}
     m^{{a}^{k}_{j}}(s). 
\end{equation}
To find a weakly minimal solution (or weakly pareto optimal solution) of the problem \eqref{model_pro}, we find a solution of the following scalar representation of (\ref{model_pro}): 
\begin{equation}\label{uscop}
 \min\limits_{s \in \mathcal{B}_k}\max\limits_{j \in [ \omega_{k}]}  \Delta_{-{K}}\left(\nabla f^{{a}^{k}_{j}}(x_{k})^{\top}s+ \tfrac{1}{2}s^{\top}\nabla^{2}f^{{a}^{k}_{j}}(x_{k})s\right).
\end{equation}

\begin{lemma}\label{weakly_minimal_point_lemma}
Any solution $s_{k} \in \mathcal{B}_k$ of the problem (\ref{uscop}) is a weakly $K^{\omega_k}$-minimal point of the vector-valued function ${\widetilde{m}}^{{a}^{k}} (s) := \left(m^{{a}^{k}_{1}}(s),  m^{{a}^{k}_{2}}(s),\ldots,m^{{a}^{k}_{\omega_k}}(s)\right)^\top$ in $\mathcal{B}_k$. 
\end{lemma}
\vspace{-0.5cm}
\begin{proof}
Let $s_{k}$ be an optimal solution of  problem (\ref{uscop}). Then, for all $s \in \mathcal{B}_k$, 
\allowdisplaybreaks
\begin{align*}
    & \max_{j \in [\omega_{k}]} \Delta_{-K}(m^{{a}^{k}_{j}}(s_{k})) \le \max_{j \in [\omega_{k}]} \Delta_{-K}(m^{{a}^{k}_{j}}(s))\\
   \implies &  0 \le \max_{j \in [\omega_{k}]} \Delta_{-K}(m^{{a}^{k}_{j}}(s))- \max_{j \in [\omega_{k}]} \Delta_{-K}(m^{{a}_{kj}}(s_k))\\ &~\le \max_{j \in [\omega_{k}]} \left(\Delta_{-K}(m^{{a}^{k}_{j}}(s)) -  \Delta_{-K}(m^{{a}^{k}_{j}}(s_k))\right)\\
   \implies & 0 \le \Delta_{-K}(m^{{a}^{k}_{j}}(s)) -  \Delta_{-K}(m^{{a}^{k}_{j}}(s_k))  \text{ for some } j \in [\omega_k] \\ 
   \implies & 0 \le \Delta_{-K} \left(m^{{a}^{k}_{j}}(s) -  m^{{a}^{k}_{j}}(s_k)\right)\text{ for some } j \in [\omega_k],    
   \text{ by Lemma \ref{proori}~(\ref{tewyueui})}\\
   \implies & m^{{a}^{k}_{j}}(s) - m^{{a}^{k}_{j}}(s_{k}) \notin - \text{int}(K) ~\text{for some}~ j \in [\omega_{k}], ~\text{by Lemma \ref{proori}~(\ref{fyf})}~\\
   \implies & m^{{a}^{k}_{j}}(s) \nprec_{K} m^{{a}^{k}_{j}}(s_{k}) ~\text{for some}~ j \in [\omega_{k}].
   \end{align*} 
Then, there exists no $s \in \mathcal{B}_k$ such that $m^{{a}^{k}_{j}}(s) \prec_K m^{{a}^{k}_{j}}(s_{k})$ for all $j \in [\omega_k]$. Therefore, $s_{k}$ is a weakly $K^{\omega_k}$-minimal point of ${\widetilde{m}}^{{a}^{k}}$ in $\mathcal{B}_k$. 
\end{proof}
As the objective function of the problem \eqref{uscop} is nonsmooth, we recast it by  
\begin{equation}\label{ghtgdyu}
\left\{\begin{aligned}
   &{\min }& & t\\
   &\text{subject to}& &\Delta_{-{K}}\left(\nabla f^{{a}^{k}_{j}}(x_{k})^{\top}s+ \tfrac{1}{2}s^{\top}\nabla^{2}f^{{a}^{k}_{j}}(x_{k})s\right) \le t,~ j = 1,2, \ldots, {\omega}_{k},\\
  &&& \lVert s \rVert \le \Omega_{k}. 
 \end{aligned} 
 \right. 
\end{equation}
However, as explained in \cite[Example 1]{Trrionglion}, to overcome the difficulty in finding a $K^{\omega_k}$-descent direction for a nonconvex problem \eqref{vop_at_x_k} at the current iterate $x_{k}$, we add a sublinear inequality in the constraint set of \eqref{ghtgdyu} as below: 
\begin{equation}\label{ghtyu}
\left\{\begin{aligned} 
&\min  & & t\\
   & \text{subject to} & &\Delta_{-{K}}\left(\nabla f^{{a}^{k}_j}(x_{k})^{\top}s+ \tfrac{1}{2}s^{\top}\nabla^{2}f^{{a}^{k}_j}(x_{k})s\right) \le t, ~j = 1,2, \ldots, {\omega}_{k},\\
 & & &\Delta_{-{K}}\left(\nabla f^{{a}^{k}_{j}}(x_{k})^{\top}s\right)  \le t, ~j =1, 2, \ldots, {\omega}_{k}, \\
 & & & \lVert s \rVert \le \Omega_{k}. 
 \end{aligned}
 \right. 
\end{equation}
Observe that the problem (\ref{ghtyu}) can be compactly rewritten as 
\begin{align}\label{dg_aux_18_01_1}
 \min\limits_{s \in \mathcal B_{k}}\max\limits_{j \in [ \omega_{k}]}\left[ \Delta_{-{K}}\left(\nabla f^{{a}^{k}_{j}}(x_{k})^{\top}s+ \tfrac{1}{2}s^{\top}\nabla^{2}f^{{a}^{k}_{j}}(x_{k})s\right),~ \Delta_{-K}(\nabla f^{{a}^{k}_{j}}(x_{k})^{\top}s) \right].
\end{align}
It is noteworthy that the objective function of the problem \eqref{dg_aux_18_01_1} is the objective function in \eqref{dg_aux_17_01_2}. Hence, any $s_k$ that is found by solving the problem (\ref{ghtyu}) essentially satisfies \eqref{dg_aux_17_01_2}.  

\subsection{Calculation of reduction ratio}\label{subsection_reduction_ratio}
Once a step $s_k$ is identified, by solving \eqref{ghtyu}, from the trust-region $\mathcal{B}_k$, we aim to check if the direction $s_k$ is a $\preceq^{l}_{K}$-descent direction for $F: \mathcal{B}_k \rightrightarrows \mathbb{R}^m$ at $x_k$; if $\preceq^{l}_{K}$-descent, then we may update the iterate by $x_{k + 1} := x_k + s_k$. To check if the identified $s_k$ is a $\preceq^{l}_{K}$-descent direction, we evaluate decrease of values of component functions in ${\widetilde{f}}^{{a}^{k}}$ for the movement of the argument point from $x_k$ to $x_k + s_k$, which can be executed by $ - \Delta_{- K}(f^{{a}^{k}_{j}}(x_{k}+s_{k})-f^{{a}^{k}_{j}}(x_{k}))$. Note that for $- \Delta_{- K}(f^{{a}^{k}_{j}}(x_{k}+s_{k})-f^{{a}^{k}_{j}}(x_{k})) > 0$, 
\[0 < - \Delta_{- K}(f^{{a}^{k}_{j}}(x_{k}+s_{k})-f^{{a}^{k}_{j}}(x_{k})) \le \Delta_{- K}(f^{{a}^{k}_{j}}(x_{k})) - \Delta_{- K}(f^{{a}^{k}_{j}}(x_{k}+s_{k})),\]
i.e., $f^{{a}^{k}_{j}}(x_{k}+s_{k}) \prec_K f^{{a}^{k}_{j}}(x_{k})$, and hence $s_k \in \mathcal{B}_k$ is a step along a $K$-descent direction of $f^{{a}^{k}_{j}}$ at $x_k$. Thus, we call the value of $- \Delta_{- K}(f^{{a}^{k}_{j}}(x_{k}+s_{k})-f^{{a}^{k}_{j}}(x_{k}))$ as a measurement of the \emph{actual reduction} of the vector-valued function $f^{{a}^{k}_{j}}$ due to $s_k$.

A \emph{predicted value of the reduction} of $f^{{a}^{k}_{j}}$ due to the step $s_k$ at $x_k$ is measured by a measurement of the movement of the function value of its model function $m^{{a}^{k}_{j}}$, which we measure by the positive value $\Delta_{-K}
({m}^{{a}^{k}_{j}}(0)-{m}^{{{a}^{k}_{j}}}(s_{k}))$. 

With the values of the measurements of the actual reduction and predicted reduction, we define a \emph{reduction ratio} $\rho^{{a}^{k}_{j}}_{k}$ as follows: for all $j \in [\omega_{k}]$, 
\begin{align*}
&\sigma^{{a}^{k}_{j}}_{k} := ~\frac{\Delta_{- K}(f^{{a}^{k}_{j}}(x_{k}+s_{k})-f^{{a}^{k}_{j}}(x_{k}))}{\Delta_{-K}
({m}^{{a}^{k}_{j}}(0)-{m}^{{{a}^{k}_{j}}}(s_{k}))}  ~\text{ and }~\rho^{{a}^{k}_{j}}_{k} :=~ \frac{\text{actual reduction}}{\text{predicted reduction}} =- \sigma^{{a}^{k}_{j}}_{k}.   
\end{align*}

\begin{proposition}\label{descveri}
Let $x_k$ be not a $K$-critical point of \eqref{fgcx}. Then, for any $j \in [\omega_k]$, a trust-region step $s_{k}$ satisfies $f^{{a}^{k}_{j}}(x_k + s_k) \prec_K f^{{a}^{k}_{j}}(x_k)$ if and only if $\rho^{{a}^{k}_{j}}_{k} >0$. 
\end{proposition}
\vspace{-0.5cm}
\begin{proof} Let $s_{k}$ satisfy $f^{{a}^{k}_{j}}(x_k + s_k) \prec_K f^{{a}^{k}_{j}}(x_k)$. Then,  $\Delta_{-K}(f^{{a}^{k}_{j}}(x_{k}+s_{k})-f^{{a}^{k}_{j}}(x_{k}) ) < 0 $. As $s_{k}$ is identified by solving (\ref{ghtyu}), we get 
\begin{align*}
\max_{ j \in [\omega_k]} \Delta_{-K}(m^{{a}^{k}_{j}}(s_{k})) ~&~ \le \max_{ j \in [\omega_k]} \left\{\Delta_{-K}(m^{{a}^{k}_{j}}(s_{k})),~\Delta_{-K}(\nabla f^{{a}^{k}_{j}}(x_{k})^{\top}s_k) \right\} \\ 
~&~ \le \max_{ j \in [\omega_k]} \left\{\Delta_{-K}(m^{{a}^{k}_{j}}(0)),~\Delta_{-K}(\nabla f^{{a}^{k}_{j}}(x_{k})^{\top}0) \right\} = 0, 
\end{align*}
i.e., $\Delta_{-K}(m^{{a}^{k}_{j}}(s_{k})) \le 0$. Thus,  
 \begin{align}\label{dg_aux_19_01_1}
     \Delta_{-K}(m^{{a}^{k}_{j}}(0)-m^{{a}^{k}_{j}}(s_{k})) \ge \Delta_{-K}(m^{{a}^{k}_{j}}(0))-\Delta_{-K}(m^{{a}^{k}_{j}}(s_{k})) \ge 0 
 \end{align}
and 
$$\rho^{{a}^{k}_{j}}_{k}= \frac{-\Delta_{- K}(f^{{a}^{k}_{j}}(x_{k}+s_{k})-f^{{a}^{k}_{j}}(x_{k}))}{\Delta_{-K}
({m}^{{a}^{k}_{j}}(0)-{m}^{{{a}^{k}_{j}}}(s_{k}))} > 0. $$

Conversely, suppose that $\rho^{{a}^{k}_{j}}_{k} > 0$. Then,  $\Delta_{-K}(f^{{a}^{k}_{j}}(x_{k} + s_{k}) - f^{{a}^{k}_{j}}(x_{k})) < 0$, and hence 
$f^{{a}^{k}_{j}}(x_{k}+s_{k}) \prec_K f^{{a}^{k}_{j}}(x_{k})$. 
\end{proof}

\begin{note}\label{suitred}
If we take $\rho^{{a}^{k}_{j}}_{k}$ as 
\begin{align}\label{wrong_rho}
\rho^{{a}^{k}_{j}}_{k} = \frac{\Delta_{-K}(f^{{a}^{k}_{j}}(x_{k})-f^{{a}^{k}_{j}}(x_{k}+s_{k}))}{\Delta_{-K}(m^{{a}^{k}_{j}}(0)-m^{{a}^{k}_{j}}(s_{k}))} \text{ for all } j \in [\omega_{k}],
\end{align}
instead of $\rho^{{a}^{k}_{j}}_{k} = - \sigma^{{a}^{k}_{j}}_{k}$, then $\rho^{{a}^{k}_{j}}_{k}$ cannot be a correct choice of the reduction ratio. This is due to the fact that if we consider the formula \eqref{wrong_rho}, then for a given $\eta_1 \in (0, 1)$, $\rho^{{a}^{k}_{j}}_{k} \ge \eta_1$ may not imply that $s_k$ is a $K$-descent direction for each of $f^{a_j^k}$'s, $j \in [\omega_k]$. 
To realize this, let us take the following simple instance: 
\[ m = 2, n = 1, p = 1, K = \mathbb{R}^2_{+} \text{ and } F(x) = \{f^1(x)\}, \]
where $f^1(x) = \left(f^{1,1}(x), f^{1,2}(x)\right)^\top$. Then, according to the formula \eqref{wrong_rho}, for a given $\eta_1 \in (0, 1)$, 
\begin{align}\label{aux_15_01}
 \rho^{{a}^{k}_{j}}_{k} \ge \eta_1  
\Longrightarrow \max\{f^{1,1}(x_k) - f^{1,1}(x_k + s_k), f^{1,2}(x_k) - f^{1,2}(x_k + s_k)\} > 0,  
\end{align}
which does not necessarily imply that $f^{1,1}(x_k + s_k) < f^{1,1}(x_k)$ and $f^{1,2}(x_k + s_k) < f^{1,2}(x_k)$. For example, consider $\eta_1 = \tfrac{1}{8}, \Omega_k = \tfrac{1}{2}, x_k = 0$,
\begin{align*}  f^{1,1}(x) = 2 \sin x - 8 \cos x - 10^4 x \sin x^2 \text{ and } f^{1,2}(x) = \sin x - \tfrac{32}{5} \cos x. \end{align*} 
Then, at $x_k$, $\omega_k = 1$, and the model functions corresponding to $f^{1,1}$ and $f^{1,2}$ are given by 
\[m^{1,1}(s_{k}) = 2s_{k} + 4s^{2}_{k} \text{ and } m^{1,2}(s_{k}) = s_{k} + \tfrac{16}{5} s^{2}_{k}, \text{ respectively}. \]
Hence, the solution to the problem \eqref{ghtyu} is obtained at $s_k = -\tfrac{5}{32}$. With this value of $s_k$, we see that 
$\rho_k^{a_j^k} = 0. 3612 \ge \eta_1$, but $f^{1,1}(x_k + s_k) = 29.9294 > -8 = f_{1,1}(x_k)$. Hence, the formula \eqref{wrong_rho} cannot be a correct choice of $\rho^{{a}^{k}_{j}}_{k}$.

It is noteworthy that once $\rho^{{a}^{k}_{j}}_{k}$ is taken as $- \sigma^{{a}^{k}_{j}}_{k}$, then \eqref{aux_15_01} gives that $f^{1,1}(x_k + s_k) < f^{1,1}(x_k)$ and $f^{1,2}(x_k + s_k) < f^{1,2}(x_k)$, i.e., $s_k$ is a $K$-descent direction of $f^{1}$  at $x_k$. 
\end{note}

  \begin{proposition}\label{vop_descent_implies_descent_sop}
If a trust-region step $s_k$ at $x_k$ is a $K^{\omega_k}$-descent direction of the objective function of \eqref{vop_at_x_k}, then $s_k$ is a $\preceq^{l}_{K}$-descent direction of $F$ at $x_k$. 
\end{proposition}

\begin{proof}
Let  $s_{k}$ be a $K^{\omega_k}$-descent direction of $\widetilde{f}^{a^{k}}$ at $x_{k}$. Then, there exists $t_{0} >0$ such that for all $t \in (0, t_{0}]$,  
\begin{align}\label{supgtdjf} 
& f^{{a}^k_{j}} (x_{k}+ ts_{k})  \prec_{K} f^{{a}^{k}_{j}} (x_{k}) ~\text{for all}~ j \in [\omega_{k}] \nonumber \\
 \implies & \{ f^{{a}^{k}_{j}} (x_{k}) \}_{j \in [\omega_{k}]} + K \subseteq \{f^{{a}^{k}_{j}}(x_{k} + t s_{k}) \}_{j \in [\omega_{k}]} + \text{int}(K) + K. 
\end{align}
Hence, we obtain 
\begin{align*}
 & F(x_{k}) \subseteq \{ f^{{a}^{k}_{j}} (x_{k}) \}_{j \in [\omega_{k}]} + K ~\text{by Proposition 2.1 in \cite{steepmethset}}\\
&~~~~~~~~\subseteq \{f^{{a}^{k}_{j}} (x_{k} + t s_{k}) \}_{j \in [\omega_{k}]} + \text{int}(K) + K ~\text{by}~(\ref{supgtdjf}) \\
&~~~~~~~~\subseteq F (x_{k}+t s_{k})+ \text{int}(K), 
\end{align*} 
i.e., $F(x_{k}+t s_{k}) \prec^{l}_{K} F(x_{k})$. Thus, $s_{k}$ is a $\preceq^{l}_{K}$-descent direction of $F$ at $x_{k}$.
\end{proof}

\subsection{Trust-region radius update}\label{subsection_tr_radius_update}
 
Based on the value of $\rho^{{a}^{k}_{j}}_{k}$, we decide acceptance of the trial step $s_{k}$ and update the trust-region radius $\Omega_{k}$.
We compare the reduction ratio $\rho^{{a}^{k}_{j}}_{k}$ with two pre-specified threshold parameters  $\eta_{1},\eta_{2}\in (0, 1)$, where $\eta_{1}$ is near to $0$ and $\eta_{2}$ is close to $1$, and classify the current iterate into one of the following three cases. \\ \\ 
\emph{Case} 1.  \textbf{(Successful iteration)}. 
We call an iteration to be successful if it satisfies the following two conditions: (i) $\rho^{{a}^{k}_{j}}_{k} \ge \eta_{1}$ for all $j \in [\omega_{k}]$, and (ii) there exists $~ l \in [\omega_{k}]$ such that $\rho^{{a}^{k}_{l}}_{k} < \eta_{2}$. If $\rho^{{a
}^{k}_{j}}_{k} \ge \eta_{1}$ for all $j \in [\omega_{k}]$, then we note that
\begin{align}\label{threshold1}
   & -\sigma^{{a}^{k}_{j}}_{k} =\frac{-\Delta_{-K}(f^{{a}^{k}_{j}}(x_{k}+s_{k})-f^{{a}^{k}_{j}}(x_{k}))}{\Delta_{-K}(m^{{a}^{k}_{j}}(0)-m^{{a}^{k}_{j}}(s_{k}))} \ge \eta_{1}, \nonumber \\
 \text{i.e.,}~ &-\Delta_{-K}(f^{{a}^{k}_{j}}(x_{k}+s_{k})-f^{{a}^{k}_{j}}(x_{k})) \ge  \eta_{1} \Delta_{-K}(-m^{{a}^{k}_{j}}(s_{k})).
 \end{align}
Since $\eta_{1} > 0$, from (\ref{threshold1}), we obtain for all $j \in [\omega_{k}]$ that $f^{{a}^{k}_{j}}(x_{k}+s_{k}) \preceq_{K} f^{{a}^{k}_{j}}(x_{k}).$
We thus have for all $k \in \mathbb N \cup \{0\}$ that 
\allowdisplaybreaks
\begin{align}
F(x_{k}) & \subseteq \{ f^{a^{k}_{1}}(x_{k}), f^{a^{k}_{2}}(x_{k}), \ldots, f^{a^{k}_{w_{k}}}(x_{k}) \} + K \nonumber \text{ by \cite[Proposition 2.1]{steepmethset}} \nonumber\\
&\subseteq \{f^{a^{k}_{1}}(x_{k}+ s_{k}), f^{a^{k}_{2}}(x_{k}+s_{k}), \ldots, f^{a^{k}_{w_{k}}}(x_{k}+s_{k})\}+ \text{int}(K) + K  \nonumber\\
 & \subseteq F(x_{k}+s_{k}) + \text{int}(K) . \label{ghdeotye}
 \end{align}
Therefore, $\{F(x_{k})\}$ is a monotonically nonincreasing sequence with respect to $\preceq^l_K$. Moreover, by Lemma \ref{proori} (\ref{tewyueui}), we have 
\begin{align}\label{dg_aux_18_01_3}
  &\Delta_{-K}(f^{{a}^{k}_{j}}(x_{k}))-\Delta_{-K}(f^{{a}^{k}_{j}}(x_{k}+s_{k}) ) \ge \eta_{1}(\Delta_{-K}(m^{{a}^{k}_{j}}(0))-\Delta_{-K}(m^{{a}^{k}_{j}}(s_{k}))). 
\end{align}
Thus, the condition $\rho^{{a}^{k}_{j}}_{k} \ge \eta_{1}$ may result in only a small reduction in $f^{a^{k}_{j}}$, which might not be entirely satisfactory.
This leads to the necessary condition of step $s_{k}$ to be not entirely satisfactory that is there exists $~ l \in [\omega_{k}]$ such that 
\begin{align}\label{ihvk_fgc}
   \Delta_{-K}(f^{{a}^{k}_{l}}(x_{k}))-\Delta_{-K}(f^{{a}^{k}_{l}}(x_{k}+s_{k})) < \eta_{2} (\Delta_{-K}(m^{{a}^{k}_{l}}(0))-\Delta_{-K}(m^{{a}^{k}_{l}}(s_{k}))),
\end{align}
where $\eta_{2} > \eta_{1}$ and $\eta_{2}$ is close to 1. As per (\ref{ihvk_fgc}), there exists $~ l \in [\omega_{k}]$ such that decrement of $f^{{a}^{k}_{l}}$ is smaller than $\eta_{2}$ times the decrement of its model $m^{a^{k}_{l}}$. Then, applying Lemma \ref{proori} (\ref{tewyueui}) in the inequality (\ref{ihvk_fgc}), we get
\begin{align*}
   - \Delta_{-K}(f^{{a}^{k}_{l}}(x_{k}+s_{k})-f^{{a}^{k}_{l}}(x_{k})) < \eta_{2} \Delta_{-K}(m^{{a}^{k}_{l}}(0)-m^{{a}^{k}_{l}}(s_{k})).
\end{align*}
This implies that the existence of an 
$l \in [\omega_{k}]$ such that $\rho^{{a}^{k}_{l}}_{k} < \eta_{2}.$
Combining the two conditions (\ref{dg_aux_18_01_3}) and (\ref{ihvk_fgc}), we can say that although the decrements of all $f^{{a}^{k}_{j}}$ are greater than $\eta_{1}$ times the decrements of $m^{{a}^{k}_{j}}$, not all decrements are bigger than $\eta_{2}$ times $m^{{a}^{k}_{j}}$. Hence, we call this step $s_{k}$ to be ``successful'' (but not ``very successful'') and 
update $x_{k}$ to $x_{k}+s_{k}$. Similar to the conventional trust-region radius update rule \cite{Trrionglion}, $\Omega_{k}$ is updated to $\Omega_{k+1}\in (\gamma_{2}\Omega_{k},  \Omega_{k}]$, where $\gamma_{2}$ is close to $1$. 

Next, we consider the case where all the objective function components get a satisfactory reduction. 

\noindent
\emph{Case} 2. \textbf{(Very successful iteration)}. 
Let the iterate $x_{k}$ satisfy  $\rho^{{a}^{k}_{j}}_{k} \ge \eta_{2}$ for all $j \in [\omega_{k}]$. In this case, we have
\begin{align*}
   &-\sigma^{{a}^{k}_{j}}_{k} = \frac{-\Delta_{-K}(f^{{a}^{k}_{j}}(x_{k}+s_{k})-f^{{a}^{k}_{j}}(x_{k}))}{\Delta_{-K}(m^{{a}^{k}_{j}}(0)-m^{{a}^{k}_{j}}(s_{k}))} > \eta_{2} \\
 \text{i.e.,}~ & \Delta_{-K}(f^{{a}^{k}_{l}}(x_{k}))-\Delta_{-K}(f^{{a}^{k}_{l}}(x_{k}+s_{k}))> \eta_{2} ({\Delta_{-K}(m^{{a}^{k}_{l}}(0))-\Delta_{-K}(m^{{a}^{k}_{l}}(s_{k}))}). 
\end{align*}
Since $\eta_{2}$ is close to $1$, every model function $m^{{a}^{k}_{j}}$ acts as a good local approximation of its associated objective function $f^{{a}^{k}_{j}}$. Therefore, this trust-region step $s_{k}$ is considered ``very successful'' for updating $x_{k}$ to $x_{k}+s_{k}$. Following the conventional update rule \cite{Trrionglion}, trust-region radius update rule $\Omega_{k}$ is enlarged to $\Omega_{k+1} \in (\Omega_{k}, \infty)$.  \\

\noindent
\emph{Case} 3. \textbf{(Unsuccessful iteration)}. 
We call an iteration to be unsuccessful if there exists an $~ l \in [\omega_{k}]$ for which the function value does not decrease, i.e.,  
\begin{align}\label{uniop_oyt}
&f^{a^{k}_{l}}(x_{k}) \prec_{K} f^{a^{k}_{l}}(x_{k}+s_{k}) \nonumber\\ 
\implies &  \Delta_{-K}(f^{a^{k}_{l}}(x_{k})) <  \Delta_{-K}(f^{a^{k}_{l}}(x_{k}+s_{k})) \text{ by Lemma \ref{proori}~(\ref{htytuttd})} \\
\implies & 0 < \Delta_{-K}(f^{a^{k}_{l}}(x_{k}+s_{k})-f^{a^{k}_{l}}(x_{k})) \text{ by Lemma  \ref{proori}~(\ref{tewyueui})}.
\end{align}
This implies that there exists an $~ l \in [\omega_{k}]$ such that $\rho^{a^{k}_{l}}_{k} < 0 \le \eta_{1}$. 
Thus, the necessary condition for the iteration $k$ to be unsuccessful and $s_{k}$ to be rejected is that there exists an $~ l \in [\omega_{k}]$ such that $\rho^{a^{k}_{l}}_{k}\le \eta_{1}$. In this case, the trust-region radius $\Omega_{k}$, similar to \cite{Trrionglion}, is contracted to $\Omega_{k+1} \in [\gamma_{1}\Omega_{k}, \gamma_{2}\Omega_{k}]$, where $0 < \gamma_{1} \le \gamma_{2} <1$.
 
 \subsection{Stopping condition}\label{subsection_stopping_condition}
Let $\theta(x_{k})$ and $s(x_{k})$ be the optimal value and optimal point, respectively, of the subproblem (\ref{ghtyu}), i.e., 
\begin{equation}\label{sdgftdu}
    \resizebox{.87\hsize}{!}{$\theta(x_k) := \min\limits_{ s \in  \mathcal B_{k}}\max\limits_{j \in [ \omega_{k}]}\left\{ \Delta_{-{K}}\left(\nabla f^{{a}^{k}_{j}}(x_{k})^{\top}s+ \tfrac{1}{2}s^{\top}\nabla^{2}f^{{a}^{k}_{j}}(x_{k})s\right) ,\Delta_{-{K}} \left( \nabla f^{{a}^{k}_{j}}(x_{k})^{\top}s\right)\right\}$}, 
   \end{equation} 
   and 
   \begin{align}\label{hdhvsd} 
   \resizebox{.9\hsize}{!}{$s(x_k) := \underset{s\in \mathcal B_{k}}{\operatorname{argmin}}\max\limits_{j \in [\omega_{k}]}\bigg\{\Delta_{-{K}}\bigg(\nabla f^{{a}^{k}_{j}}(x_{k})^{\top}s+ \tfrac{1}{2}s^{\top}\nabla^{2}f^{{a}^{k}_{j}}(x_{k})s\bigg), \Delta_{-{K}} \bigg( \nabla f^{{a}^{k}_{j}}(x_{k})^{\top}s\bigg)\bigg\}.$}
\end{align} 

\begin{theorem}\label{critopti}For
the functions $\theta$ and $s$ in \eqref{rbdwbyew} and \eqref{hdhvsd}, respectively, the following results hold. 
\begin{enumerate}[(a)]
\item The mapping $ \theta$ is continuous at any $x \in \mathcal{R}$, where $\mathcal{R}$ is the set of all regular point for \eqref{fgcx}, and well-defined for all $x \in \mathbb R^{n}$.
\item\label{sucfr} 
The following three statements satisfy $(i) \Longleftrightarrow(ii) \Longrightarrow (iii)$: \begin{enumerate}[(i)]
    \item $x_{k}$ is not a $K$-critical point for \eqref{fgcx};
    \item\label{gpourer} $\theta(x_{k})< 0$;
    \item $ s(x_{k}) \neq {0}$.
\end{enumerate}
\end{enumerate}
 \end{theorem}
\begin{proof}
\begin{enumerate}[(a)]
\item 
Let $\bar x \in \mathcal{R}$ and $\varepsilon > 0$. Since $\bar x$ is a regular point of $F$, there exists a neighbourhood $U$ of $\bar x$ such that for all $z \in U$, we have $\omega(z) = \bar{\omega} \text{ and }  P_{z} \subseteq P_{\bar x}.$ 
For any given $x \in \mathbb{R}^{n}$, $\bar a = (\bar{a}_{1}, \bar{a}_{2}, \ldots, \bar{a}_{\bar{\omega}}) \in P_{z}$ and $j \in [\bar{\omega}]$, we define two functions $\phi_{x,\bar{a}_{j}}: \mathbb R^{n} \to \mathbb R$ and $ \psi_{x,\bar{a}_{j}}: \mathbb R^{n} \to \mathbb R$ by   
\begin{align*}
\phi_{x,\bar{a}_{j}}(z) & := \Delta_{-{K}}\left(\nabla f^{\bar{a}_{j}}(z)^{\top}s(x)+ \tfrac{1}{2}s(x)^{\top}\nabla^{2}f^{\bar{a}_{j}}(z)s(x)\right), \\
\text{ and } \psi_{x, \bar{a}_{j}}(z) & := \Delta_{-{K}} \left( \nabla f^{\bar{a}_{j}}(z)^{\top}s(x)\right), \text{ respectively}. 
\end{align*}
Then, we find that $ \theta(z) = \min\limits_{(\bar{a},s) \in P_{z} \times {\mathcal B}} ~ \max\limits_{j \in [\bar{\omega}]} \left\{\phi_{x, \bar{a}_{j}}(z),  \psi_{x, \bar{a}_{j}}(z)\right\}$.  
Let $\mathcal{W} \subseteq \mathbb{R}^n$ be a compact set containing $\bar x$. We note by Lemma \ref{proori} \eqref{aeruin} for any $z \in U \cap \mathcal{W}$ that 
\begin{align}\label{trrjfg}
    ~&\max_{j\in [\bar{\omega}]} ~ \lvert \phi_{{x}, \bar{a}_{j}}(z)- \phi_{{x}, \bar{a}_{j}}(\bar{x}) \rvert \nonumber \\
  \le ~& \max_{j \in [\bar{\omega}]} \resizebox{.86\hsize}{!}{$\left\| 
 \left(\nabla f^{\bar{a}_{j}}(z)^{\top}s( {x})+ \tfrac{1}{2}s({x})^{\top}\nabla^{2}f^{\bar{a}_{j}}(z)s({x})\right)- \left(\nabla f^{\bar{a}_{j}}(\bar{x})^{\top}s({x})+ \tfrac{1}{2}s({x})^{\top}  \nabla^{2}f^{\bar{a}_{j}}(\bar{x})s({x})\right)\right\|$} \nonumber \\
 \le ~&  \lVert s({x}) \rVert  \max_{j \in [\bar{\omega}]} \lVert \nabla f^{\bar{a}_{j}}(z)-\nabla f^{\bar{a}_{j}}(\bar{x}))\rVert + \tfrac{1}{2} \lVert s({x}) \rVert^{2} \max_{j \in [\bar{\omega}]} \lVert \nabla^{2} f^{\bar{a}_{j}}(z)-\nabla^{2} f^{\bar{a}_{j}}(\bar{x}))\rVert. 
\end{align}
Since $ f^{\bar{a}_{j}}$ is twice continuously differentiable, the functions $\nabla f^{\bar{a}_{j}}$ and $\nabla^{2} f^{\bar{a}_{j}}$ are uniformly continuous in  $\mathcal{W}$. Hence, for the given $\varepsilon$, there is $\delta_{1} >0 $ such that  $\lVert z-\bar{x}\rVert < \delta_{1}$ and $z \in U \cap \mathcal{W}$ imply $\lVert \nabla f^{\bar{a}_{j}}(z)-\nabla f^{\bar{a}_{j}}(\bar{x}))\rVert\le \frac{\varepsilon}{2{\Omega}_{\max}}$. Similarly, there is $\delta_{2} >0 $ such that  $\lVert z-\bar{x}\rVert < \delta_{2}$ and $ z \in U \cap\mathcal{W}$ imply $\lVert \nabla^{2} f^{\bar{a}_{j}}(z)-\nabla^{2} f^{\bar{a}_{j}}(\bar{x}))\rVert\le \frac{\varepsilon}{{\Omega}^2_{\max}}$. 
Let $\delta := \min\{ \delta_{1}, \delta_{2}\} >0 $.
As $ \lVert s({x})\rVert \le {\Omega_{\max}}$, from (\ref{trrjfg}), we have 
\begin{align*}
\lVert z-\bar{x}\rVert < \delta \text{  and }
z \in U \cap\mathcal{W} ~\Longrightarrow~ \max_{j \in [\bar{\omega}]}\lvert \phi_{{x}, \bar{a}_{j}}(z)- \phi_{{x}, \bar{a}_{j}}(\bar{x}) \rvert \le \varepsilon.
\end{align*} 
This means that $\{\phi_{{x}, \bar{a}_{j}}: j \in  [{\bar{\omega}}]\}$ is equicontinuous at $\bar x$. Similarly, $\{\psi_{{x}, \bar{a}_{j}}: j \in  [{\bar{\omega}}]\}$ is equicontinuous at $\bar x$. Therefore, the family $\{\Phi_x\}_{x \in \mathbb{R}^n}$, where \begin{align*}
\Phi_{{x}}(z) := \max\limits_{j \in [ \bar{\omega}]} \{ \phi_{{x},\bar{a}_{j}}(z), \psi_{{x}, \bar{a}_{j}}(z)\}, ~z \in U \cap \mathcal{W}, 
\end{align*}
is equicontinuous. Then, for the given $ \varepsilon >0$, there exists $\bar \delta > 0$ such that for all $ z \in U \cap \mathcal W$ satisfying $\lVert z-\bar x \rVert < \delta$, we have $
    \lvert \Phi_{{x}}(z)-\Phi_{{x}}(\bar x)\rvert < \varepsilon   \text{ for all } z \in \mathbb{R}^{n}.$ 
Hence, for all $z \in U \cap \mathcal W$ satisfying $ \lVert z-\bar{x} \rVert < \bar \delta$, it follows that 
\begin{align*}
\theta(z) \le   ~& \Phi_{\bar x}(z) \le \Phi_{\bar x}(\bar x)+ \lvert \Phi_{\bar x}(z)-\Phi_{\bar x}(\bar{x})\rvert < \theta(\bar{x})+ \varepsilon, 
\end{align*}
i.e., $\theta(z)-\theta(\bar x) < \varepsilon$. By altering the role of $z$ and $\bar{x}$, we find that $ \lvert \theta(\bar{x})-\theta(z) \rvert < \varepsilon$. Thus, the continuity of $\theta$ at $\bar x$ follows. Hence, $\theta$ is continuous in $\mathcal{R}$.

As for any $x \in \mathbb{R}^n$, the value of $\theta(x)$ is obtained by minimization of a maximum function of continuous functions over the compact set $P_{x} \times \mathcal{B}$, the function $\theta$ is well-defined.

\item \fbox{(i) implies (ii).}  
Since $x_{k}$ is not $K$-critical for (\ref{fgcx}), there exists $a^{k} \in P_{k}$ and $ s_{k} \in \mathcal{B}_k$ such that \begin{align}\label{vdfsfgs}
      & \Delta_{-K}(\nabla f^{{a}^{k}_{j}}(x_{k})^{\top}{s_k})<0  \text{ for all } j \in [\omega_{k}]. 
     \end{align}

\noindent
\emph{Case}1. Let  for all $j \in [\omega_{k}]$, $ \Delta_{-K}({{s_k}}^{\top}\nabla^{2}f^{{a}^{k}_{j}}(x_{k}) {s_k}) \le 0$. Then, Lemma \ref{proori} (\ref{tewyueui}), implies that 
\begin{equation} \label{tydss} \Delta_{-K}\left(\nabla f^{{a}^{k}_{j}}( x_{k})^{\top}{s_{k}}+ \tfrac{1}{2}{{s_{k}}}^{\top}\nabla^{2}f^{{a}^{k}_{j}}(x_{k}) {s_{k}} \right) < 0 \text{ for all } j \in [\omega_k]. 
\end{equation}
Combining \eqref{vdfsfgs} and \eqref{tydss}, we get $\theta(x_k) < 0$. \\

\noindent
\emph{Case }2. Let $\Delta_{-K}( {{s_{k}}}^{\top}\nabla^{2}f^{{a}^{k}_{l}}(x_{k}) {s_{k}} ) >0$ for some $l \in [\omega_{k}]$. Then, by considering $s'_{k} := \alpha s_{k}$, with $\alpha >0$ small enough, we 
have 
\begin{equation}\label{erfte}
  \Delta_{-K}\left(\nabla f^{{a}^{k}_{j}}(x_{k})^{\top}{s'_{k}}\right)+ \tfrac{1}{2} \Delta_{- K}\left({{s'_{k}}}^{\top}\nabla^{2}f^{{a}^{k}_{j}}(x_{k}) {s'_{k}}\right) < 0 \text{ for all } j \in [\omega_{k}]. 
\end{equation}
For instance, any $\alpha$ that satisfies  
$0 < \alpha < \frac{- 2\Delta_{-K}(\nabla f^{{a}^{k}_{l}}(x_k)^{\top}s_k)}{\Delta_{-K}(s_{k}^{\top}\nabla^{2} f^{{a}^{k}_{l}}(x_{k})s_{k})}  \text{ for all } l \in [\omega_k]$ with $\Delta_{-K}( {{s_{k}}}^{\top}\nabla^{2}f^{{a}^{k}_{l}}(x_{k}) {s_{k}} ) >0$ holds  \eqref{erfte}.
As $s_k$ lies in $\mathcal{B}_k$, a small enough choice of $\alpha$ leads to $s'_k \in \mathcal{B}_k$. 
With such an $s'_k \in \mathcal{B}_k$, we note from \eqref{vdfsfgs} that 
\begin{equation}\label{aux_deb_g1}
\Delta_{-K}(\nabla f^{{a}^{k}_{j}}(x_{k})^{\top}{s'_k}) = \alpha \Delta_{-K}(\nabla f^{{a}^{k}_{j}}(x_{k})^{\top}{s_k})<0  \text{ for all } j \in [\omega_{k}].
\end{equation} 
From \eqref{erfte} and Lemma \ref{proori} (\ref{tewyueui}), 
\begin{align}\label{aux_deb_g2}
\Delta_{-K}\left(\nabla f^{{a}^{k}_{j}}( x_{k})^{\top}{s'_{k}}+ \tfrac{1}{2}{{s'_{k}}}^{\top}\nabla^{2}f^{{a}^{k}_{j}}(x_{k}) {s'_{k}} \right) < 0 \text{ for all } j \in [\omega_k].    
\end{align}
In view of (\ref{aux_deb_g1}) and (\ref{aux_deb_g2}), we get $\theta(x_{k})<0$.

\fbox{(ii) implies (i).} If $\theta(x_k)<0$, then there exists a feasible point $(t_k, s_k)$, with $t_k < 0$, of \eqref{ghtyu}  such that 
       $\Delta_{- K}(\nabla f^{{a}^{k}_{j}}(x_k)^{\top} s_{k})\le t_k <0 ~\text{for all}~ j \in [\omega_{k}].$
   So, $x_k$ is not $K$-critical.

\fbox{(ii) implies (iii).} If $s(x_k) = {0}$, then from the definition of $s(x_k)$ and \eqref{sdgftdu}, we get $\theta(x_{k}) =0$. This contradicts $\theta(x_{k}) <0$. So, $ s(x_{k}) \neq 0$.  
\end{enumerate}

\end{proof}

\begin{algorithm}[H]
\caption{Trust-region algorithm for solving (\ref{fgcx}) }\label{avds1}
\begin{small}
\begin{enumerate}[1:]
\item \emph{Input and initialization} \\
Provide the functions $f^i: \mathbb{R}^n \to \mathbb{R}^m$, $i = 1, 2, \ldots, p$, of the problem (\ref{fgcx}). \\ 
Choose an initial point $x_{0}$, an initial trust-region radius $\Omega_{0}$, a maximum allowed trust-region radius $\Omega_{\max}$, a tolerance value $\varepsilon$ for stopping criterion, threshold parameters $ \eta_{1}, \eta_{2} \in (0,1)$ with $\eta_1 < \eta_2$, and fractions $\gamma_{1}, \gamma_{2}\in (0,1)$ to shrink the trust-region radius. \\ 
Set the iteration counter $k := 0$.

\item \label{gtesv} \emph{Computation of minimal elements}  \\ 
Compute $M_{k} := \text{Min}({F}(x_{k}), K) = \{r_1,r_2,\ldots,r_{\omega_k}\}$. \\
Compute the partition set $P_{k} := I_{r_1}\times I_{r_2}\times\cdots\times  I_{r_{\omega_k}}$. \\ 
Find $p_k := \lvert P_{k} \rvert$ and $\omega_{k} := \lvert \text{Min}({F}(x_{k}), K)\rvert$.

\item \label{step3_choice_of_a} \emph{Choice of an `${a}^{k}$' from $P_{k}$} \\ 
Select an element ${a}^{k} := ({a}^{k}_{1},{a}^{k}_{2}, \ldots, {a}^{k}_{\omega_{k}}) \in P_{k}$ such that 
$({a}^{k},  s_{k}) \in \underset{(a, s)\in P_{k}\times \mathcal B_{k}}{\operatorname{argmin}} \Theta_{x_k}(a, s).$

\item \emph{Model definition}\\
Compute the model functions $m^{{a}^{k}_{j}}$ for all $j \in [{\omega}_{k}]$ by the formula 
\begin{align*}
    m^{{a}^{k}_{j}}(s)=\nabla f^{{a}^{k}_{j}}(x_{k})^{\top}s+\tfrac{1}{2}s^{\top} \nabla^{2}f^{{a}^{k}_{j}}(x_{k})s, ~ \lVert s \rVert \le \Omega_{k}. 
\end{align*}

\item \label{step size calculation}\emph{Step size calculation} \\ 
Find $(t_{k}, s_{k})$ by solving the subproblem (\ref{ghtyu}).

\item \emph{Termination criterion} \label{tocri} \\ 
If $\lvert t_{k} \rvert < \varepsilon$, then stop and provide $x_{k}$ as an approximate $K$-critical point of (\ref{fgcx}).\\ 
Else, proceed to the next step. 

\item \emph{Calculation of reduction ratio}\label{cal_reduct} \\ 
Execute \begin{align*}
        \sigma^{{a}^{k}_{j}}_{k} := \frac{\Delta_{- K}(f^{{a}^{k}_{j}}(x_{k}+s_{k})-f^{{a}^{k}_{j}}(x_{k}))}{\Delta_{-K}({m}^{{a}^{k}_{j}}(0)-{m}^{{a}^{k}_{j}}(s_{k}))} \text{ and }~\rho^{{a}^{k}_{j}}_{k} := - \sigma^{{a}^{k}_{j}}_{k}~ \text{for all} ~j \in [\omega_{k}].~  
\end{align*}

\item \emph{Acceptance of the computed step size} \label{getwqs} \\ 
If $\rho^{{a}^{k}_{j}}_{k}\ge \eta_{1}$ for all $j \in [{\omega}_{k}]$, then update   $x_{k+1} := x_{k}+ s_{k}.$

\item \emph{Rejection of the computed step size}\label{rej_ste} \\    
If there exists $j \in [\omega_{k}]$ such that $\rho^{{a}^{k}_{j}}_{k} < \eta_{1}$, then $x_{k+1} := x_{k}$.

\item \emph{Trust-region radius update}\label{trust_region_radius_update} \\ 
Choose $\Omega_{k+1}$ by the rule 
\begin{align*}
\begin{cases}
(\gamma_{2} \Omega_{k}, \Omega_{k}] & \text{if}~ \rho^{{a}^{k}_{j}}_{k} \ge \eta_1 ~\forall~ j \in [{\omega}_{k}]~ \text{and}~ \exists~ l\in [\omega_{k}] ~\text{such that}~\rho^{{a}^{k}_{l}}_{k} < \eta_{2} ~  \text{(Successful)} \\
(\Omega_{k}, \infty) & \text{if}~  \rho^{{a}^{k}_{j}}_{k} \ge \eta_2 ~\forall ~j\in [{\omega}_{k}]~ (\text{Very successful})\\
[\gamma_{1}\Omega_{k}, \gamma_{2}\Omega_{k}] &\text{if}~ \exists~ l \in [\omega_{k}]~\text{such that}~ \rho^{{a}^{k}_{l}}_{k} <  \eta_{1} ~ (\text{Unsuccessful}). 
\end{cases}
\end{align*}
\item \emph{Move to next iteration}\\  
$ k := k+1$ and go to Step \ref{gtesv}.
\end{enumerate}
\end{small}
\end{algorithm}

\subsection{Well-definedness of Algorithm \ref{avds1}}\label{welldefine_trust_algori} The 
well-definedness of Algorithm \ref{avds1} depends upon Step  \ref{step3_choice_of_a}, Step \ref{step size calculation} and Step \ref{cal_reduct}. Below, we provide reasons behind the well-definedness of these three steps.

In Step \ref{step3_choice_of_a}, we choose a specific $(a^{k}, s_k)$ from $P_{k} \times \mathcal{B}_k$ as defined in (\ref{dg_aux_17_01_2}). The existence of such an $(a^k, s_k)$ is assured by the following reason: 
\begin{enumerate}[(i)]
\item 
 for each $a^k \in P_k$, the function $\Theta_{x_k}(a^k, \cdot)$ as defined in \eqref{namfun} is a continuous function, 

\item the set $\mathcal{B}_k$ is a compact set, and 

\item 
$P_k$ is a finite set. 
\end{enumerate}

In Step \ref{step size calculation}, we solve the subproblem  (\ref{ghtyu}) to find a solution $(t_{k}, s_{k})$. For this step to be well-defined, it is required that a solution exists for the subproblem (\ref{ghtyu}). From Lemma \ref{proori}~(\ref{aeruin}), we recall that $\Delta_{-K}$ is a continuous function. Thus,  the subproblem (\ref{ghtyu}) is a minimization of a continuous function over a compact subset of $\{s\in \mathbb R^{n}: \lVert s \rVert \le \Omega_{k}\}$. Hence, a solution $(t_{k}, s_{k})$ to (\ref{ghtyu}) is assured in Step \ref{step size calculation}.

In Step \ref{cal_reduct}, we calculate the reduction ratio $\rho^{{a}^{k}_{j}}_{k}$. When we enter Step \ref{cal_reduct}, then note that $x_k$ is not a $K$-critical point of \eqref{fgcx}. Indeed, because if $x_k$ is a $K$-critical point, then by Proposition \ref{nrwelyd}, $t_k = \theta(x_k) = 0$, and the algorithm will then get terminated at Step \ref{tocri}. At a not $K$-critical point $x_k$,  $\Delta_{-K}(m^{{a}^{k}_{j}}(0)-m^{{a}^{k}_{j}}(s_{k}))$ is positive. Then, the denominator in  $\rho^{{a}^{k}_{j}}_{k}$ is nonzero and hence Step  \ref{cal_reduct} is well-defined. 

\section{Global convergence analysis}\label{global_conv_ana}
In this section, we discuss the global convergence of Algorithm \ref{avds1} under the following assumptions.

\begin{assumption} \label{Assumption_1}The level set  ${\mathcal{L}_{0}}: = \{ x \in \mathbb R^{n}: F(x) \preceq^{l}_{K} F(x_{0})\}$ is bounded. 
\end{assumption}

\begin{assumption}\label{fun_bndbelow}

The function $f^{i}$ is bounded below for each $i \in [p]$. 
\end{assumption}

\begin{assumption}
There exists a $\mathcal{K}_{1}>0$ such that for each $ i \in [p]$, 
\[
\lVert \nabla f^{{i}}(x) \rVert \le \mathcal{K}_{1} \text{ for all } x \in \mathbb{R}^n. \]
\label{grad_bnd}
\end{assumption}

\vspace{-0.75cm}
\begin{assumption}\label{hess_bnd} 
There exists a $\mathcal{K}_{2} > 0$ such that for each $i \in [p]$, 
\[\lVert \nabla^{2} f^{i}(x) \rVert \le \mathcal{K}_{2} ~\text{for all} ~ x \in \mathbb{R}^n.\]
\end{assumption}

\begin{theorem}\label{fsgcsb}
Let $x_{k}$ be not a $K$-critical point of \eqref{fgcx}, and $v \in \mathbb R^{n}$ be a $K$-descent direction of $f^{a^{k}_j}$ at $x_{k}$ for all $j \in [\omega_k]$. Then, for each $j \in [\omega_{k}]$, there exists $\bar{t}_j >0$ satisfying $\lVert {\bar{t}_j} v \rVert \le \Omega_{k}$ such that 
\begin{align*}
\Delta_{-K}(m^{{a}^{k}_{j}}\left({\bar{t}_j}v\right)) \le  \Delta_{-K}(m^{{a}^{k}_{j}}\left(tv\right)) \text{ for all }  {t} \ge 0 ~\text{satisfying}~\lVert {t} v\rVert \le \Omega_{k}.  
\end{align*}
\noindent
In addition, for each $j \in [\omega_{k}]$, 
\small{\begin{align*}
  \Delta_{-K}(m^{{a}^{k}_{j}}(0)- m^{{a}^{k}_{j}}\left({\bar{t}_j}v\right))\ge -\tfrac{1}{2} \frac{\Delta_{-K}(\nabla f^{a^{k}_{j}}(x_{k})^{\top} v) }{\lVert v \rVert} \min\bigg\{-\frac{\Delta_{-K}(\nabla f^{a^{k}_{j}}(x_{k})^{\top}v)}{\lVert  v \rVert \mathcal{K}_{2}}, \Omega_{k} \bigg\}.
 \end{align*}
}
\end{theorem}

\begin{proof}
Let $x_{k}$ be not a $K$-critical point of \eqref{fgcx}. We define a function $\phi_j: \left[0, \frac{\Omega_{k}}{\lVert v \rVert }\right] \to \mathbb{R}$ by 
    $\phi_j(t) := \Delta_{-K}(m^{{a}^{k}_{j}}(tv)).$ 
Since $ \Delta_{-K}$ is  continuous on $\mathbb{R}^m$  and $m^{{a}^{k}_{j}}$ is continuous on $\mathbb R^{n}$, $\phi_j$ is also continuous on the compact set $\left[0, \frac{\Omega_{k}}{\lVert v \rVert}\right]$. So, there exists $\bar t_j $ in $\left[0, \frac{\Omega_{k}}{\lVert v \rVert}\right]$ such that $\phi_j(\bar t_j) \le \phi_j (t)$ for all $t$ in $\left[0, \frac{\Omega_{k}}{\lVert v \rVert}\right]$. 
We show that $\bar t_j \neq 0$. Then, the first part of the result will be followed.

As $x_k$ is not a $K$-critical point of \eqref{fgcx}, in the lines of the proof of \eqref{aux_deb_g2}, there exists $t_j'>0$ such that $t'_j v =: s_k' \in \mathcal{B}_k$ satisfies $
 \Delta_{-K}(m^{a^k_j} (s_k')) < 0,~   
\text{i.e., }  \phi_j(t'_j) < 0.  $
Hence, $\phi_j(\bar t_j) \le \phi_j(t_j') < 0 = \phi_{j}(0)$, and thus $\bar t_j \ne 0$.

For the second part of the result, we pick $t^{\ast}_j := \frac{\Delta_{-K}(- \nabla f^{{a}^{k}_{j}}(x_{k})^{\top} v)}{ \Delta_{-K}(v ^{\top} \nabla^{2}f^{{a}^{k}_{j}}(x_{k}) v ) }.$ 
Here, there are two possible cases. \\

\noindent
\fbox{Case 1.} Let $\Delta_{- K}(v^{\top} \nabla^{2}f^{{a}^{k}_{j}}(x_{k})v) >0$. In this case, by Lemma \ref{proori} (\ref{tewyueui}), we see that 
\begin{align*}
 t^{\ast}_j = \frac{\Delta_{-K}(- \nabla f^{{a}^{k}_{j}}(x_{k})^{\top} v)}{ \Delta_{-K}(v ^{\top} \nabla^{2}f^{{a}^{k}_{j}}(x_{k}) v ) } \ge -\frac{\Delta_{- K}( \nabla f^{{a}^{k}_{j}}(x_{k})^{\top}v) }{\Delta_{- K  }( v^{\top} \nabla^{2}f^{{a}^{k}_{j}}(x_{k}) v)} >0.   
\end{align*}
Here, we have two further subcases. \\ 

\noindent
\fbox{Subcase 1 (i).} Suppose that $t^{\ast}_jv$ lies within the trust-region, i.e., $\lVert t^{\ast}_j v  \rVert \le \Omega_{k}$. We choose $$\tilde t_j := -\frac{\Delta_{- K}( \nabla f^{{a}^{k}_{j}}(x_{k})^{\top}v) }{\Delta_{-K}(v^{\top} \nabla^{2}f^{{a}^{k}_{j}}(x_{k}) v)} \le \frac{\Delta_{- K}(- \nabla f^{{a}^{k}_{j}}(x_{k})^{\top} v) }{ \Delta_{- K}(v^{\top} \nabla^{2}f^{{a}^{k}_{j}}(x_{k})v)} = t^{\ast}_j. $$ Then,  
\allowdisplaybreaks
 \begin{align*}
   &~ \Delta_{- K}(m^{{a}^{k}_{j}}(\bar t_j v)) \le \Delta_{- K}(m^{{a}^{k}_{j}}(\tilde t_j v)) \\
   = &~  \Delta_{-K}(\nabla f^{{a}^{k}_{j}}(x_{k})^{\top} \tilde{t}_j v+\tfrac{1}{2} (\tilde{t}_j v)^{\top}\nabla^{2}f^{{a}^{k}_{j}}(x_{k})(\tilde{t}_j v)) \text{ by Lemma \ref{proori} \eqref{ubrsyu} and \eqref{tewyueui}}\\
   \le &~ \tilde{t}_j \Delta_{- K}( \nabla f^{{a}^{k}_{j}}(x_{k})^{\top}v)+ \tfrac{1}{2} \tilde{t}_j^{2} \Delta_{- K}(v^{\top}\nabla^{2}f^{{a}^{k}_{j}}(x_{k})v) \\
   = &~ -\frac{\Delta_{- K}( \nabla f^{{a}^{k}_{j}}(x_{k})v)}{ \Delta_{-K}(v^{\top} \nabla^{2}f^{{a}^{k}_{j}}(x_{k}) v)} \Delta_{-K}(\nabla f^{{a}^{k}_{j}}(x_{k}) v)
 \\
&~+ \tfrac{1}{2}\left(\tfrac{\Delta_{-K}(\nabla f^{{a}^{k}_{j}}(x_{k})^{\top}v)}{{\Delta_{-K}(v^{\top} \nabla^{2}f^{{a}^{k}_{j}}(x_{k}) v)} } \right)^{2}{\Delta_{-K}(v^{\top} \nabla^{2}f^{{a}^{k}_{j}}(x_{k})v)} = -\tfrac{1}{2} \frac{(\Delta_{- K} (\nabla f^{{a}^{k}_{j}}(x_{k})^{\top}v)^{2} }{\Delta_{-K}(v^{\top} \nabla^{2}f^{{a}^{k}_{j}}(x_{k}) v)}.
 \end{align*}
From Lemma \ref{proori} (\ref{aeruin}),  we get 
\begin{align*}
     \lvert \Delta_{- K}( v^{\top} \nabla^{2} f^{{a}^{k}_{j}}(x_{k})v) -\Delta_{-K}(0)\rvert \le {\lVert v \rVert^{2} \lVert \nabla^{2} f^{{a}^{k}_{j}}(x_{k}) \rVert }\le \mathcal{K}_{2} \|v\|^2.
\end{align*}
Now, using the fact $ \Delta_{- K}(m^{{a}^{k}_{j}}_{k}(0))=0$, we obtain
 \begin{align}\label{tyduire}
    - \Delta_{- K}(m^{{a}^{k}_{j}}_{k}(\bar t_j v)) \ge \tfrac{1}{2} \frac{ (\Delta_{- K} (\nabla f^{{a}^{k}_{j}}(x_{k})^{\top} v))^{2}}{\Delta_{-K}(v^{\top} \nabla^{2}f^{{a}^{k}_{j}}(x_{k}) v)}  
     \ge \tfrac{1}{2} \frac{( \Delta_{- K} (\nabla f^{{a}^{k}_{j}}(x_{k})^{\top} v))^{2}}{{\lVert v \rVert}^2 \mathcal{K}_{2}}.
 \end{align}

\noindent
\fbox{Subcase 1 (ii).} Let $t^{\ast}_jv$ lie outside the trust-region, i.e., $\lVert t^{\ast}_j v \rVert  > \Omega_{k}$. In this subcase, we take $\tilde{t}_j := \frac{\Omega_{k}}{\lVert v \rVert}$. Then,  $\Delta_{-K}(m^{{a}^{k}_{j}}_{k}(\bar t_j v)) \le \Delta_{- K}(m^{{a}^{k}_{j}}_{k}(\tilde t_j v))$, and hence   
\begin{align}\label{dg_aux_0809_1}
   & \Delta_{-K}(m^{{a}^{k}_{j}}_{k}(\bar t_j v)) \le \tfrac{\Omega_{k}}{\lVert v \rVert} \Delta_{-K} ( \nabla f^{{a}^{k}_{j}}(x_{k})^{\top} v)+ \tfrac{1}{2} \left(\tfrac{\Omega_{k}}{\lVert v \rVert}\right)^{2} \Delta_{-K} (v^{\top} \nabla^{2}f^{{a}^{k}_{j}}(x_{k})v).
    \end{align}
From  $ - \frac{\Delta_{- K}(\nabla f^{{a}^{k}_{j}}(x_{k})^{\top} v)}{\Delta_{- K}( v^{\top} \nabla^{2} f^{{a}^{k}_{j}}(x_{k})v)} \lVert v \rVert = \| {t}^*_j v\| > \Omega_{k}= \tilde t_j\lVert v \rVert$, we have $$ \frac{\Delta_{-K}(\nabla f^{{a}^{k}_{j}}(x_{k})^{\top} v) }{\tilde t_j} < -\Delta_{- K}(v^{\top} \nabla^{2} f^{{a}^{k}_{j}}(x_{k})v). $$
Therefore, from \eqref{dg_aux_0809_1}, we get 
\begin{align}\label{weteh}
  \Delta_{- K}(m^{{a}^{k}_{j}}_{k}(0))- \Delta_{- K}(m^{{a}^{k}_{j}}_{k}(\bar t_j v)) \ge & -\frac{\Omega_{k}}{\lVert v \rVert}\Delta_{- K} ( \nabla f^{{a}^{k}_{j}}(x_{k})^{\top} v) \\
  +\tfrac{1}{2} \left(\frac{\Omega_{k}}{\lVert v \rVert}\right)^{2} \frac{\Delta_{- K}(\nabla f^{{a}^{k}_{j}}(x_{k})^{\top}v) }{\tilde t_j} =&  -\tfrac{1}{2}\frac{\Omega_{k}}{\lVert v \rVert}\Delta_{- K} (\nabla f^{{a}^{k}_{j}}(x_{k})^{\top} v). \nonumber
\end{align}

\noindent
\fbox{Case 2.} Let $\Delta_{- K}(v^{\top} \nabla^{2} f^{{a}^{k}_{j}}(x_{k})v) \le 0 $. In this case, taking $ \tilde{t}_j := \frac{\Omega_{k}}{\lVert v \rVert} $, we see that  
\begin{align*}
     \Delta_{-K}(m^{{a}^{k}_{j}}(\bar t_j v)) \le \Delta_{-K} (m^{{a}^{k}_{j}}(\tilde t_j v)) \le &~ \frac{\Omega_{k}}{\lVert v \rVert} \Delta_{- K} ( \nabla f^{{a}^{k}_{j}}(x_{k})^{\top}v)\\
    + \tfrac{1}{2} \left(\frac{\Omega_{k}}{\lVert v \rVert }\right)^{2}\Delta_{- K} (v^{\top} \nabla^{2} f^{{a}^{k}_{j}}(x_{k})^{\top} v) \le &~ \frac{\Omega_{k}}{\lVert v \rVert}\Delta_{- K} ( \nabla f^{{a}^{k}_{j}}(x_{k})^{\top} v).
\end{align*}
Then, 
\begin{align}\label{pterhd}
  \Delta_{-K}(m^{{a}^{k}_{j}}(0) )-\Delta_{- K}(m^{{a}^{k}_{j}}(\bar t_j v))  \ge -\tfrac{1}{2}\frac{\Omega_{k}}{\lVert v \rVert}\Delta_{- K} ( \nabla f^{{a}^{k}_{j}}(x_{k})^{\top} v).\end{align}
Accumulated from (\ref{tyduire}), (\ref{weteh}), and (\ref{pterhd}), we have
\begin{align*}
 &~ \Delta_{-K}(m^{{a}^{k}_{j}}(0)-m^{a^{k}_{j}}(\bar t_j v)) \\
 \ge &~ \Delta_{-K}(m^{{a}^{k}_{j}}(0) )-\Delta_{-K}(m^{{a}^{k}_{j}}(\bar t_j v)) \text{ by Lemma \ref{proori} (\ref{tewyueui})} \\ 
 \ge &~ -\tfrac{1}{2} \left(\frac{ \Delta_{- K} (\nabla f^{{a}^{k}_{j}}(x_{k})^{\top} v)}{\lVert v \rVert} \right)\min \left\{ 
 -\frac{\Delta_{-K}(\nabla f^{{a}^{k}_{j}}(x_{k})^{\top} v)}{\lVert v \rVert\mathcal K_{2}},\Omega_{k}\right \} > 0, 
\end{align*}
which completes the proof.
\end{proof}

\begin{cor}\label{tncdfd} If $s_{k}$ is a solution of the subproblem \eqref{ghtyu}, and $x_{k}$ is not a $K$-critical point of \eqref{fgcx}, then there exists a positive constant $\beta$ such that  for all $j \in [\omega_{k}]$, $\Delta_{-K}(m^{{a}^{k}_{j}}(0)) \ge \Delta_{-K} (m^{{a}^{k}_{j}}(s_{k}))$ and  
 \begin{align*}
    & \Delta_{-K}(m^{{a}^{k}_{j}}(0)- m^{{a}^{k}_{j}}\left(s_{k}\right)) \ge -\tfrac{\beta}{2} \frac{\Delta_{-K}(\nabla f^{{a}^{k}_{j}}(x_{k})^{\top} s_{k})}{\lVert s_{k} \rVert}  \min\left\{-\tfrac{\Delta_{-K}(\nabla f^{{a}^{k}_{j}}(x_{k})^{\top}s_{k})}{\lVert  s_{k} \rVert \mathcal{K}_{2}}, \Omega_{k}  \right\}.  
 \end{align*}   
\end{cor}

\begin{proof}
Straightforward. 
\end{proof}

\begin{theorem}\label{approx_obj_model} If $x_{k}+ s_{k} \in \mathcal{B}_{k}$, then for all $j\in [{\omega}_{k}]$,  
\begin{align*}
    \lvert \Delta_{-K}(f^{{a}^{k}_{j}}(x_{k}+ s_{k})- m^{{a}^{k}_{j}}(x_{k}+s_{k})) \rvert \le \mathcal{K}  \Omega^{2}_{k}. 
\end{align*}
\end{theorem}

\begin{proof}
For each $j \in [{\omega}_{k}]$, as $f^{{a}^{k}_{j}}$ is twice continuously differentiable, by Assumption \ref{hess_bnd}, we get  
\begin{align}\label{ytehf_1}
    f^{{a}^{k}_{j}}(x_{k}+s_{k}) = f^{{a}^{k}_{j}}(x_{k}) + \nabla f^{{a}^{k}_{j}}(x_{k})^{\top}s_{k} + \|s_{k}\|^2 O(1). 
\end{align}
Likewise, for the model function $m^{{a}^{k}_{j}}$, we have  
\begin{align}\label{ytehf_2}
  m^{{a}^{k}_{j}}(x_{k}+s_{k})=f^{{a}^{k}_{j}}(x_{k})+  \nabla f^{{a}^{k}_{j}}(x_{k})^{\top}s_{k} + \|s_{k}\|^2 O(1). 
\end{align}
From (\ref{ytehf_1}) and  (\ref{ytehf_2}), we obtain 
\begin{align}\label{yvxure}
   & \lvert \Delta_{-K}(f^{{a}^{k}_{j}}(x_{k}+ s_{k})- m^{{a}^{k}_{j}}(x_{k}+s_{k}) )\rvert \notag \\ 
   \le ~ & \|f^{{a}^{k}_{j}}(x_{k}+ s_{k})- m^{{a}^{k}_{j}}(x_{k}+s_{k}) \| \text{ by Lemma \ref{proori} (\ref{aeruin})} \notag \\ 
   \le~ &  \mathcal{K} \Omega^{2}_{k} \text{ for some } \mathcal{K} > 0. 
\end{align}
This concludes the proof. 
\end{proof}

\begin{theorem}\label{rfgrwe} 
Let $x_k$ be a not a $K$-critical point of \eqref{fgcx}. At $x_{k}$, if $s_{k}$ is a solution of the subproblem \eqref{ghtyu} with $\Delta_{-K}(\nabla f^{{a}^{k}_{j}}(x_{k})^{\top}  s_{k}) \neq 0$, and \begin{align}\label{fddstere}
    \Omega_{k} \le \frac{\beta \lvert \Delta_{- K}(\nabla f^{{a}^{k}_{j}}(x_{k})^{\top}s_{k}) \rvert(1-\eta_{2})}{2 \mathcal{K}_{2}\lVert s_{k}\rVert}
\end{align}
for all $ j \in [\omega_{k}]$, then the iteration $k$ is successful. 
\end{theorem}

\begin{proof}
Since $ 0 < \eta_{2}<  1$ and  $ 0 < \beta < 1$, we have $\frac{\beta}{2}(1-\eta_{2}) < 1$. Thus, from  (\ref{fddstere}), we have
\begin{align}\label{yputrdv_uio}
       \Omega_{k} <  \frac{\lvert\Delta_{- K} (\nabla f^{{a}^{k}_{j}}(x_{k})^{\top}s_{k}) \rvert} {\mathcal{K}_{2}\lVert s_{k}\rVert} \text{ for all } j \in [\omega_{k}].
   \end{align}
Since $x_{k}$ is not $K$-critical, from (\ref{yputrdv_uio}) and Corollary \ref{tncdfd}, we obtain for all $j \in [\omega_{k}]$ that  
\begin{align}
    \Delta_{- K}(m^{{a}^{k}_{j}}(0)-m^{{a}^{k}_{j}}(s_{k})) & \ge \Delta_{-K}(m^{{a}^{k}_{j}}(0))-\Delta_{-K}(m^{{a}^{k}_{j}}(s_{k})) \label{tyerypi} \\
   & \ge \tfrac{\beta}{2}\frac{\lvert\Delta_{-K}(\nabla f^{{a}^{k}_{j}}(x_{k})^{\top}s_{k})\rvert}{\lVert s_{k}\rVert } \Omega_{k} >0.\label{dg_aux2_0809} 
    \end{align}
We claim that the $k$-th iteration is successful. On the contrary, let us assume that the $k$-th iteration is unsuccessful. Then, there exists at least one $l \in [\omega_{k}]$  such that $\rho^{a^{k}_{l}}_{k} < \eta_{1}$. From (\ref{uniop_oyt}), this implies that 
\begin{align}\label{oijkl_outj}
 - \Delta_{-K}(f^{a^{k}_{l}}(x_{k}+s_{k})-f^{a^{k}_{l}}(x_{k})) < 0.   
\end{align}
 From the definition of reduction ratio $\rho^{a^{k}_{l}}_{k}$ and 
(\ref{oijkl_outj}), we get 
\begin{align}\label{yhtsd}
    \rho^{{a}^{k}_{l}}_{k}-1= ~~& \frac{-\Delta_{-K}((f^{{a}^{k}_{l}}(x_{k}+s_{k}) -f^{{a}^{k}_{l}}(x_{k}))}{ \Delta_{-K}(-m^{{a}^{k}_{l}}(s_{k}))} -1  \nonumber\\
\overset{(\ref{tyerypi})}{\ge} ~~&\frac{-\Delta_{-K}(f^{{a}^{k}_{l}}(x_{k}+s_{k})-f^{a^{k}_{l}}(s_{k}))+\Delta_{-K}(m^{a^{k}_{l}}(s_{k}))}{\Delta_{-K}(m^{{a}^{k}_{l}}(0)) -\Delta_{- K}(m^{{a}^{k}_{l}}(s_{k})}.
\end{align}
Using (\ref{yhtsd}) and Lemma \ref{proori}(\ref{tewyueui}), we then have $\lambda_{j} \in (0,1)$ such that $\xi^{j}_{k} := x_{k}+\lambda_{j}s_{k}$ satisfies
{\small{\begin{align*}
 1-\rho^{{a}^{k}_{l}}_{k} \le ~& \frac{\Delta_{-K}\left(\frac{1}{2}s^{\top}_{k}\nabla^{2}f^{a^{k}_{l}}(x_{k})s_{k}-\frac{1}{2}s^{\top}_{k}\nabla^{2}f^{a^{k}_{l}}(\xi^{j}_{k})s_{k}\right)}{(\Delta_{-K}(0)-\Delta_{- K}(m^{{a}^{k}_{l}}(s_{k}))} \\  
\le ~&\tfrac{1}{2} \frac{\lVert s_{k}\rVert^{2} \lVert 
\nabla^{2} f^{{a}^{k}_{j}}(x_{k})-\nabla^{2} f^{{a}^{k}_{j}}(\xi^{j}_{k}) \rVert}{(-  \Delta_{- K}(m^{{a}^{k}_{j}}(s_{k}))}   
\overset{(\ref{dg_aux2_0809})}{\le} \frac{2\Omega_{k} \|s_k\| \mathcal{K}_{2}}{\beta \lvert \Delta_{-K}( \nabla f^{{a}^{k}_{j}}(x_{k})^{\top}s_{k}) \rvert}  
 \overset{(\ref{fddstere})} <  1-\eta_{2},  
\end{align*}}}
which gives $\rho^{{a}^{k}_{l}}_{k} > \eta_{2} > \eta_1$. 
This contradicts $\rho^{{a}^{k}_{l}}_{k} <  \eta_{1}$, and thus, it must be true that  $\rho^{a^{k}_{j}}_{k} \ge \eta_{1}$ for all $j \in [\omega_{k}]$. Hence, the $k$-th iteration is successful.  
\end{proof} 

Theorem \ref{rfgrwe} will now help us to establish the $K$-criticality of the limit points of the sequence $\{x_k\}$ of iterates generated by Algorithm \ref{avds1} when the number of successful iterates in $\{x_k\}$ is finite. For this purpose, let us first denote the collection of all successful iterations of  Algorithm \ref{avds1} by 
 \begin{align}\label{ytefg}
   \mathcal{S} := \{ k \in \mathbb{N} : \rho^{{a}^{k}_{j}}_{k} \ge \eta_{1}~\forall~ j \in [\omega_{k}] ~\text{and}~ \exists~ l~ \in [\omega_{k}] ~\text{for which}~\rho^{{a}^{k}_{l}}_{k}< \eta_{2}\}.
\end{align}

\begin{theorem}\label{yrbetd} Let $\{x_k\}$ be a sequence of iterates generated by Algorithm $\ref{avds1}$. If there are only finitely many iterations in $\mathcal{S}$, then there exists $x^*$ such that $x_{k}= x^{\ast}$ for all sufficiently large $k$, and $x^{\ast}$ is a $K$-critical point of \eqref{fgcx}. 
\end{theorem}

\begin{proof}
Let $k_{0}$ be the last successful iteration, and every iteration after that is unsuccessful. Then, by Step \ref{rej_ste} of Algorithm \ref{avds1}, we get all the future unsuccessful iterates as  $x_{k_{0}+1}=x_{k_{0}+q}$ for all $q > 1$. Since all iterations are unsuccessful after $k_0$, by the trust-region radius update of unsuccessful iterations  of Step \ref{trust_region_radius_update} of Algorithm \ref{avds1}, we get $\{\Omega_{k}\}$ converging to zero, i.e., for every $\epsilon >0$, there exists $\bar k > k_0$ such that 
\begin{equation}\label{dg_aux2_0909}
\Omega_{k} < \epsilon \text{ for all } k > \bar{k}.
\end{equation} 
Note that $x_{k_{0}+1}=x_{k_{0}+2} = x_{k_{0}+3} =\cdots = x_{\bar k} = x_{\bar k + 1} = \cdots$, and let us denote these iterates by $x^{\ast}$. If possible, let $x^* = x_{\bar k + 1}$ be not $K$-critical for \eqref{fgcx}. Then, by Definition \ref{spcritic}, we get  for all $k > \bar k$ that 
\begin{align}\label{dg_aux1_0909}
\Delta_{-K}(\nabla f^{{a}^{k +1}_{j}}(x_{k +1})^{\top}s_{k +1}) < 0  \text{ for all } j \in [\omega_{k +1}].
\end{align} 
From (\ref{dg_aux1_0909}), we have 
\begin{align}\label{tybj_iotf}
 \sup\limits_{k > \bar k}\left\{\frac{\Delta_{-K}(\nabla f^{{a}^{k +1}_{j}}(x_{ k +1})^{\top}s_{ k +1})}{\lVert s_{k+1} \rVert}\right\} \le 0.   \end{align}
 However, $\sup\limits_{k > \bar k} \left\{  \frac{\Delta_{-K}(\nabla f^{a^{k+1}_{j}}(x_{k+1})^\top s_{k+1})}{\lVert s_{k+1} \rVert}\right\}$ can never be zero since due to $x_{\bar k+1} = x_{\bar k+2} $ $= x_{\bar k +3} = \cdots = x^*$, $\bigg\{\frac{\Delta_{-K}(\nabla f^{a^{k+1}_{j}}(x_{k+1})^\top s_{k+1})}{\lVert s_{k+1} \rVert}: k > \bar k \bigg\}$ is a singleton set. \\ 
Letting $\epsilon := \frac{\beta}{2\mathcal K_{2}}\left\lvert\sup\limits_{k>\bar k}\left\{\frac{\Delta_{-K}(\nabla f^{{a}^{k +1}_{j}}(x_{ k +1})^{\top}s_{ k +1})}{\lVert s_{k+1} \rVert}\right\}\right\rvert(1-\eta_{2}) \overset{(\ref{tybj_iotf})}{>} 0$, we get 
\begin{align*}
    \forall j \in [\omega_{\bar k+1}]:~ \Omega_{\bar k +1} \overset{\eqref{dg_aux2_0909}}{<}~&\frac{\beta(1-\eta_{2})}{2\mathcal K_{2}}\left\lvert\sup\limits_{k>\bar k}\left\{\frac{\Delta_{-K}(\nabla f^{{a}^{k +1}_{j}}(x_{ k +1})^{\top}s_{ k +1}}{\lVert s_{k+1} \rVert}\right\}\right\rvert \\ 
\overset{\eqref{tybj_iotf}}{\le} ~&\frac{\beta \lvert \Delta_{- K}(\nabla f^{{a}^{\bar k+1}_{j}}(x_{\bar k+1})^{\top}s_{\bar k+1}) \rvert(1-\eta_{2})}{2 \mathcal{K}_{2}\lVert s_{\bar k+1}\rVert}.
\end{align*} 
As a consequence, by Theorem \ref{rfgrwe}, $(\bar k + 1)$-th iteration is a successful iteration. However, this is in contradiction to $x_{k_{0}}$ being the last successful iterate. This proves that $x^* = x_{\bar k + 1}$ is a $K$-critical point of \eqref{fgcx}. 
\end{proof}

\begin{lemma}\label{utdrb}
Suppose that $\{x_k\}$ is a sequence generated by Algorithm \ref{avds1}, and there are infinitely many successful iterations $k \in \mathcal{S}$. Then, we have 
\begin{align}\label{gdgydrg}
        \liminf_{k \to \infty}\max_{j \in [{\omega_{k}}]} \left\{ \frac{\lvert \Delta_{-K}( \nabla f^{{a}^{k}_{j}}(x_{k})^{\top}s_{k}) \rvert}{\lVert s_{k} \rVert} \right\}=0.
\end{align}
\end{lemma}
\begin{proof}
On the contrary, let us assume that (\ref{gdgydrg}) is not true. This implies that there exists $ \epsilon >0$ such that for all $k$ we have a $j \in [\omega_{k}]$ with
\begin{align}\label{tsffsw}
    \frac{\lvert \Delta_{-K}( \nabla f^{{a}^{k}_{j}}(x_{k})^{\top}s_{k}) \rvert}{\lVert s_{k} \rVert}\ge \epsilon. 
\end{align}
Now, let $k$ be a successful iteration and $(t_{k}, s_{k} )$ be the solution of the subproblem (\ref{ghtyu}). Then, for all $j \in [\omega_{k}]$, from (\ref{ytefg}), we have 
   \begin{align*}
   & \rho^{{a}^{k}_{j}}_{k}= \frac{-\Delta_{- K}( -(f^{{a}^{k}_{j}}(x_{k})- f^{{a}^{k}_{j}}(x_{k+1})))}{(\Delta_{- K}(m^{{a}^{k}_{j}}(0)- m^{{a}^{k}_{j}}(s_{k})))} \ge \eta_{1},
\end{align*}
and using Lemma \ref{proori}~(\ref{tewyueui}), we get
$\frac{\Delta_{- K}(f^{{a}^{k}_{j}}(x_{k}))-\Delta_{-K}(f^{{a}^{k}_{j}}(x_{k+1}))}{\Delta_{- K}(m^{{a}^{k}_{j}}(0)- m^{{a}^{k}_{j}}(s_{k}))} \ge \eta_{1}.$
Then, from the sufficient decrease condition for the model $m^{a^{k}}$ of $f^{a^{k}}$ in Theorem \ref{fsgcsb}, we have
\begin{align}\label{iuothr_ioew}
 & \Delta_{- K}(f^{{a}^{k}_{j}}(x_{k}))-\Delta_{- K}(f^{{a}^{k}_{j}}(x_{k+1})) \nonumber\\
  \ge   ~ & ~\eta_{1} \beta \bigg(-\tfrac{1}{2}\frac{\Delta_{- K}(\nabla f^{{a}^{k}_{j}}(x_{k})^{\top}s_{k})}{\lVert s_{k}\rVert} 
  \min \left\{ -\frac{\Delta_{- K}(\nabla f^{{a}^{k}_{j}}(x_{k})^{\top}s_{k})}{\lVert s_{k}\rVert \mathcal K_{2}}, \Omega_{k}\right\}\bigg) \nonumber\\
  \ge ~&~  \eta_{1}\epsilon \beta \tfrac{1}{2} \min \left\{ \frac{\epsilon}{\mathcal{K}_{2}}, \Omega_{\min} \right\}.
\end{align}
For every successful iteration, it holds $x_{k+1}=x_{k}+s_{k}$. Thus, summing over all the successful iterations $k$, we get, from Lemma \ref{proori} (\ref{tewyueui}) that 
\vspace{-0.4cm}
  \begin{align}\label{uyrtf_rre}
      \Delta_{-{K}}( f^{{a}^{k}_{j}}(x_{0})-f^{{a}^{k}_{j}}(x_{k+1})) ~\ge~  & \sum_{l=0, l\in \mathcal{S}}^{k} (\Delta_{- K}(f^{{a}^{k}_{j}}(x_{l}))-\Delta_{- K}(f^{{a}^{k}_{j}}(x_{l+1}))) \nonumber\\
    ~{\ge}~ & \sum_{l=0, l\in \mathcal{S}}^{k} \eta_{1}\epsilon \beta \tfrac{1}{2}\min \left\{ \frac{\epsilon}{\mathcal{K}_{2}}, \Omega_{\min} \right\} ~\text{by}~ (\ref{iuothr_ioew}) \nonumber\\
 ~ \ge~  & \alpha_{k} \eta_{1} \epsilon
    \beta \tfrac{1}{2} \min \left\{ \frac{\epsilon}{\mathcal{K}_{2}}, \Omega_{\min} \right\},
\end{align}
where $\alpha_{k}$ is the number of successful iterations up to the $k$-th iteration. Now, for infinitely many successful iterations, we have
$ \lim_{k \to \infty} \alpha_{k} = \infty$,
which, from (\ref{uyrtf_rre}) and Definition 2.1 in \cite{Ansari2019}, means that either $\lim\limits_{k \to \infty}(f^{{a}^{k}_{j}}(x_{0}))-(f^{{a}^{k}_{j}}(x_{k+1}))= \infty$ or $-K = \emptyset$. The former is in contradiction with Assumption \ref{fun_bndbelow}, and the latter contradicts the fact that $K$ is a solid cone. Therefore, there cannot exist any $\epsilon >0$ satisfying (\ref{tsffsw}), which concludes the proof. 
\end{proof}

\begin{theorem}\label{global_conv_resul}
Let $\{x_{k}\}$ be a sequence of points generated by Algorithm \ref{avds1}, and any limit point of $\{x_{k}\}$ be a regular point of $F$. If $\{ x_{k}\}$ contains infinitely many successful iterations, then $\{ x_{k}\}$ has a subsequence that converges to a $K$-critical point for \eqref{fgcx}. 
\end{theorem}  
   
\begin{proof}Without loss of generality (if required take a subsequence), we assume that all the terms of $\{x_{k}\}$ are unsuccessful iterations. Recall that the level set, $\mathcal{L}_{0} := \{x \in \mathbb R^{n}: F(x) \preceq^{l}_{K}  F(x_{0})\}$ of $F$ is bounded. Since Algorithm \ref{avds1} is a $\preceq^{l}_{K}$-decreasing method, we have $x_{1} \in \mathcal{L}_{0}$.
By induction, we prove that  $x_{k+1} \in \mathcal{L}_{0}$ if $x_{k} \in \mathcal{L}_{0}$ for all $k >1$. Using (\ref{ghdeotye}), 
 we have 
 \begin{align*}
\forall ~k\in \mathbb N \cup \{0\}:  F(x_{k+1}) = F(x_{k}+s_{k})\preceq^{l}_{K}F(x_{k})\preceq^{l}_{K}F(x_{0}).
 \end{align*}
 Since $\mathcal{L}_{0}$ is bounded, $\{ x_{k}\}$ is then a bounded sequence, and hence it 
 has a convergent subsequence $\{ x_{k_r}\}$ which converges to $\bar{x}$ say. Since $F$ is continuous, $\mathcal{L}_{0}$ is closed; hence, $\bar x\in \mathcal{L}_{0}$.

 Since $x_{k}$ is a successful iterate for $(\ref{fgcx})$, we have
   $s_{k} = x_{k+1}-x_{k}$, and thus $\{s_{k}\}_{k \in \mathcal{K}}$ is bounded, where $\mathcal K := \{ r_k\}$. So, $\{s_{k}\}_{k\in \mathcal K}$ must have a convergent subsequence $\{ s_{k}\}_{k \in \mathcal K'\subseteq \mathcal K}$, which converges to $\bar s$, say, i.e.,  $s_{k} \xrightarrow{k\in\mathcal{K}'} {\bar s} $.

 Since there is only a finite number of subsets of $[p]$ and $\{x_{k}\}$ is a sequence of regular points of $F$, by Lemma \ref{regularity_condition} and for all sufficiently large $k \in \mathcal{K}'$, we have
$\omega_{k} = \bar{\omega}, a^{k} = \bar{a}.$

  Since $\{x_{k}\}$ and $\{ s_{k}\}$ are generated by Algorithm \ref{avds1}, 
  and also due to the fact that $f^{{a}^{k}_{j}}$ is continuous for all $j\in [\omega_{k}]$, we obtain from Lemma \ref{utdrb} that 
 \small{ \begin{align*}
0 = \liminf_{k\xrightarrow{\mathcal{K'}} \infty}\max_{j \in [{\omega_{k}}]}\left\{ \frac{\lvert \Delta_{-K}( \nabla f^{{a^{k}_{j}}}({x_{k}})^{\top}{s_{k}}) \rvert}{\lVert {s_k} \rVert} \right\} 
\le   ~& \Delta_{-K} \left(\max_{j \in [\bar{\omega}]}\left\{ \frac{ (\nabla f^{\bar{a}_{j}}(\bar x)^{\top}\bar{s})}{\lVert \bar{s} \rVert}  \right\}\right), 
\end{align*} }
 i.e., 
$ \max\limits_{j \in [\bar{\omega}]}\left\{ \frac{ ( \nabla f^{\bar{a}_{j}}(\bar x)^{\top}\bar{s}) }{\lVert \bar{s} \rVert}  \right\}\notin -K$ by (\ref{hryx}) and (\ref{ufbmvfd})  of Lemma \ref{proori}.
 As a result, for  $\bar{s} \in \mathbb R^{n}$, there exists $j \in [\bar{\omega}]$ such that $\nabla f^{\bar{a}_{j}}(\bar x)^{\top}\bar{s} \notin -K$, i.e., 
\begin{align}\label{hgtre_ouyf}
 \Delta_{- K}(\nabla f^{\bar{a}_{j}}(\bar x)^{\top}\bar{s})\ge 0. \end{align}
On the other hand, since $x_{k} \xrightarrow{k\in\mathcal{K'}} \bar x$ and $\theta$ is a continuous map, we have $\theta(\bar x ) \le 0.$
 Thus,
\begin{align*}
0 \le &~\Delta_{-K}(\nabla f^{\bar{a}_{j}}(\bar x)^{\top}\bar{s}) ~\text{by}~(\ref{hgtre_ouyf}) \\
     \le & ~\max_{j \in [\bar{\omega}]} \left\{\Delta_{-K}(\nabla f^{\bar{a}_{j}}(\bar x)^{\top}\bar{s}), \Delta_{-K}\left(\nabla f^{\bar{a}_{j}}(\bar x)^{\top}\bar{s}+\frac{1}{2} {\bar{s}}^{\top}\nabla^{2} f^{\bar{a}_{j}}(\bar x)\bar{s}\right)\right\} \\
     = &~ \Theta_{\bar x}(\bar{a},\bar{s})  =\min_{(\bar{a}, s) \in P_{\bar x}\times \mathbb R^{n}}\Theta_{\bar x}(a,s)
     =  ~\theta(\bar x)\le 0,
\end{align*}
which implies $\theta(\bar x) =0$.
Therefore, from Theorem \ref{critopti},  $\bar x$ must be a $K$-critical point of \eqref{fgcx}. This completes the proof.
\end{proof}

\section{Numerical experiments}\label{Numerical_Experiment}
In this section, we present the numerical performance of the proposed Algorithm \ref{avds1} on some test SOPs given in Table \ref{test_prob_table} and compare it with the existing Steepest Descent (SD) method \cite{steepmethset} and the Conjugate Gradient Method (CGM) \cite{kumar2024nonlinear} using the basic conjugate descent rule with $\beta^{\text{CGM}}_{k} := 0.99(1-\sigma)\beta^{\text{CD}}_{k}$, where $\beta^{\text{CD}}_{k}$ is given by (49) in \cite{kumar2024nonlinear}.  For comparison, we use the performance profile proposed by Dolan and Mor{\'e} \cite{dolan2002benchmarking}. We run numerical experiments using Matlab 2023b on a system having Apple M2 chip with $8$ CPU cores and 8 GB of RAM.  

The parameters used for the three algorithms are given in Table \ref{tab2}. The performance profile is studied using the cone $K := \mathbb R^{m}_{+}$ for all $17$ problems.  

\begin{table}[!h]
\label{tab2}
\centering 
\resizebox{0.75\textwidth}{!}{\begin{minipage}{\textwidth}
\caption{Parameters used in the experiments.}
\begin{tabular}{c c c c c c c c c c c c c c c c}
\hline
$\rho $ & $\sigma$ &  $\beta_\text{SD}$ & $\nu$& $\Omega_0$ &  $\Omega_\text{max}$ & $\epsilon$ & $\eta_1$ & $\eta_2$ & $\gamma_1$ & $\gamma_2$    \\ 
\hline 
$0.0001$ & $0.1$ & $0.0001$ &$0.5$&$1$ & $20$ & $10^{-3}$  & $0.001$ & $0.75$ & $0.4$ & $0.9$  \\
\hline
\end{tabular} 
\end{minipage}}
\end{table}

First, we consider one problem (Example \ref{test_b_inst_53}) and show the performance of the three algorithms for three different ordering cones $K_1 := \mathbb{R}^2_+$, $K_{2}:=\{(y_{1}, y_{2})\in \mathbb R^{2}: 5y_{2}\ge 7 y_{1}, y_{2} \le 7 y_{1} \}$ and $ K_{3}:=\{(y_{1}, y_{2}) \in \mathbb R^{2}: y_{1} \ge 3y_{2}, 3y_{1}\le y_{2}\}$.

\begin{example}\label{test_b_inst_53} 
We consider a modified Test Instance 5.3 in \cite{steepmethset}. For $i \in [100]$, the function $f^{i}: \mathbb R^{2} \to \mathbb R^{2}$ is defined by 
\begin{align*}
f^{i}(x_1, x_2) : = 
\begin{pmatrix}
e^{\frac{x_1}{2}}\cos x_{2} + x_{1} \cos x_{2} \sin \tfrac{\pi(i-1)}{50} - x_{2} \sin x_{2} \cos^3 \tfrac{\pi(i-1)}{50} \\
e^{\frac{x_{2}}{20}} \sin x_{1} + x_{1} \sin x_{2} \sin^{3} \tfrac{\pi(i-1)}{100} + x_{2} \cos x_{2} \cos \tfrac{\pi(i-1)}{100})
\end{pmatrix}
\end{align*}
within the region $\mathcal S : = [-20, 20] \times [-20,20]$. We ran three algorithms for $100$ random initial points from $\mathcal S$. In Table \ref{tab61}, we show the performance of each algorithm using different ordering cones. From the table, we observe that no single algorithm performs the best in all aspects of the number of generated convergent sequences, average step length across iterations, number of iterations, and time taken before a sequence is converged.

In Figure \ref{Figure_Iter_Points_Image_Space_Example_53}, we display the sequence of iterates and the associated iterative points in the image space, respectively, generated by TRM for the initial point $x_{0}=(9,8)$.

\begin{table}[!h]
\label{tab61}
\centering
 \resizebox{\textwidth}{!}{\begin{minipage}{\textwidth}
 \caption{Performance of TRM, SD and CGM on Example \ref{test_b_inst_53} with different ordering cones.}
\begin{tabular}{c |c| c c c c c}
\hline
Cone & Method & No. of & 
No. of iterations              
    &  CPU time (in sec) & Average  \\
& & convergence     & (Min, Max, Mean, Median)     &  (Min, Max, Mean, Median) & Step length \\ 
\hline
 \multirow{3}{*}{\STAB{\rotatebox[origin=c]{0}{$K_{1}$}}} &$\text{TRM}$  & $48$ & $(1,50,6.95,3)$ & $(9.37, 748.18,73.17, 33.83)$  & $0.54$\\ 
&$\text{SD}$ &  $45$  &  $(1,63, 13.45,7)$  & $(0.768, 1109.56, 44.66,  3.00)$ & $0.339$
   \\ 
&$\text{CGM}$   &  $44$   &  $(1,54,10.65,8)$  & $( 0.733,415.24, 45.21, 1.51)$ & $7896.36$\\ 
\hline 
\multirow{3}{*}{\STAB{\rotatebox[origin=c]{0}{$K_{2}$}}}&
$\text{TRM}$   & $50$  & $(1, 87, 10.94,3)$ & $(8.42,
            1225.94,117.47, 36.12)$ & $0.685$\\
&$\text{SD}$  & $93$  & $(0,7,  0.65,0)$  & $(8.92, 5836.2,    321.36, 41.17)$ & $0.001$ \\
&$\text{CGM}$   &  $97$   & $(0,1,0.05,0)$   & $(3.64,20660.34,559.74,21.43)$ & $0.333$  \\
\hline 
\multirow{3}{*}{\STAB{\rotatebox[origin=c]{0}{$K_{3}$}}}&
$\text{TRM}$  &  $46$   &   $(1,69,9.65,4)$ & $(7.19,961.28, 90.98, 26.66)$ & $0.680$ \\
&$\text{SD}$  & $100$  & $(  0,    62,0.65,0)$   & $(7.02, 1158.22,        51.38, 40.39)$ & $0.01$
   \\
  &$\text{CGM}$ & $99$ & $(0,0,0,0 )$  & $(17.37, 9916.05,   175.12, 74.45)$
& $0$ \\
\hline
\end{tabular}
\end{minipage}}
\end{table}

\vspace{-0.5cm}

\begin{figure}[H]
\centering
\mbox{
\subfigure[For $K_{1}$]{\includegraphics[scale=0.24]{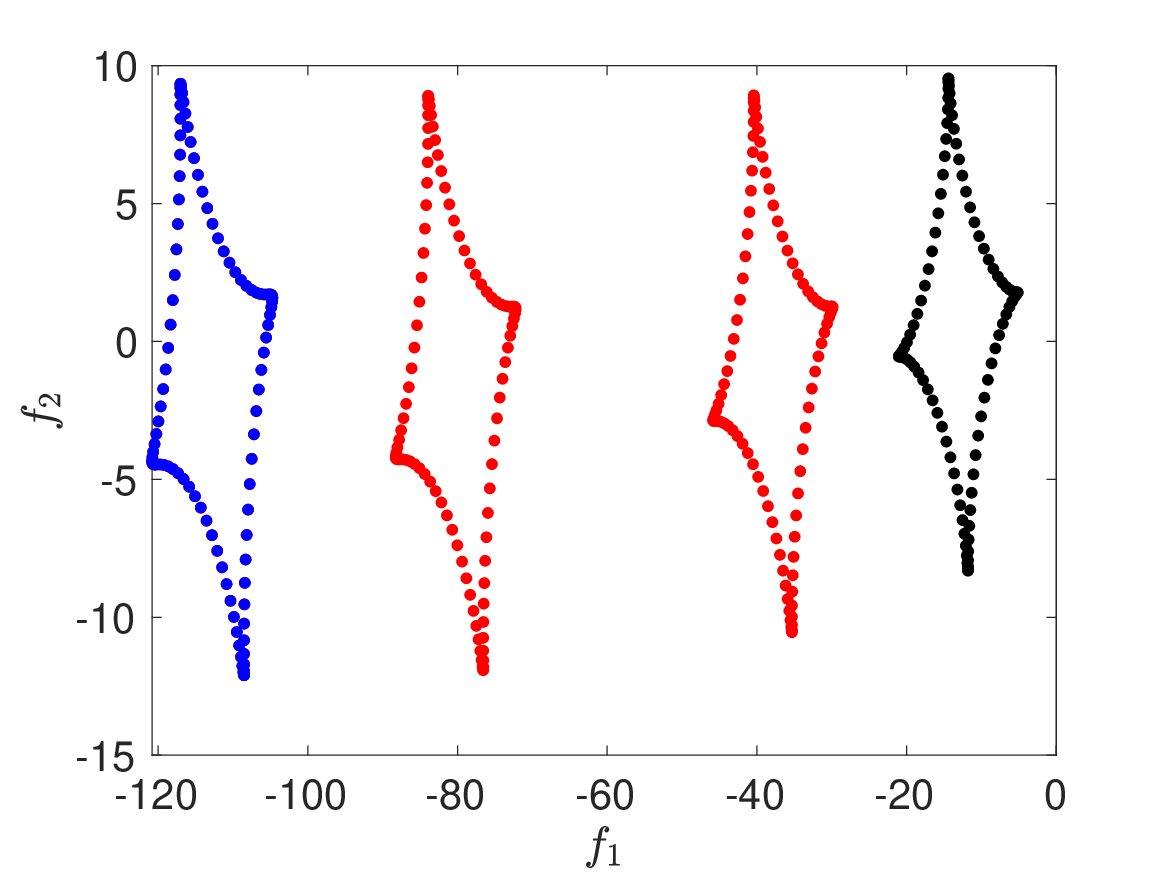}\label{Iteratve_points_image_space_53sec_cone_K1}}\quad
\subfigure[For $K_{2}$]{\includegraphics[scale=0.24]{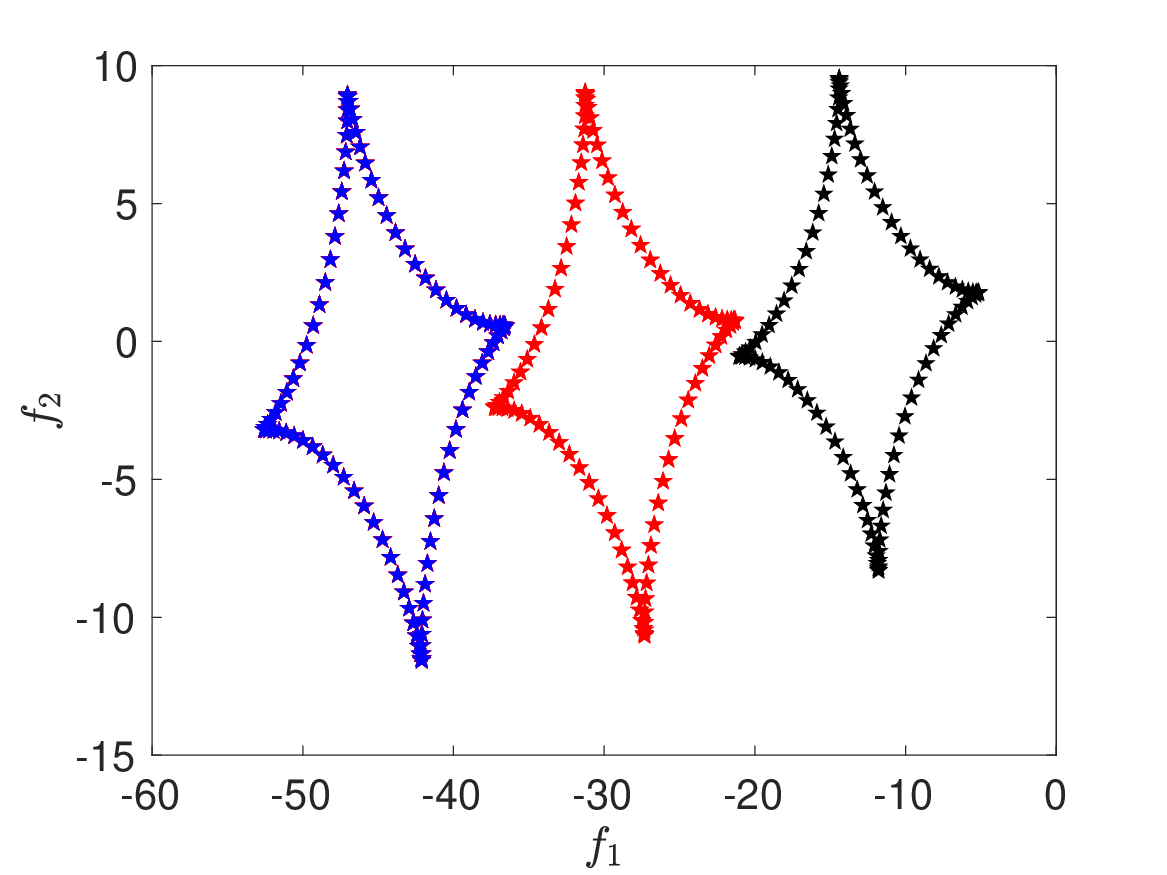}
\label{Iterative_points_image_space_53_second789_cone_K2}
\label{scnwlp_43stet}
}\quad
\subfigure[For $K_{3}$]{\includegraphics[scale=0.24]
{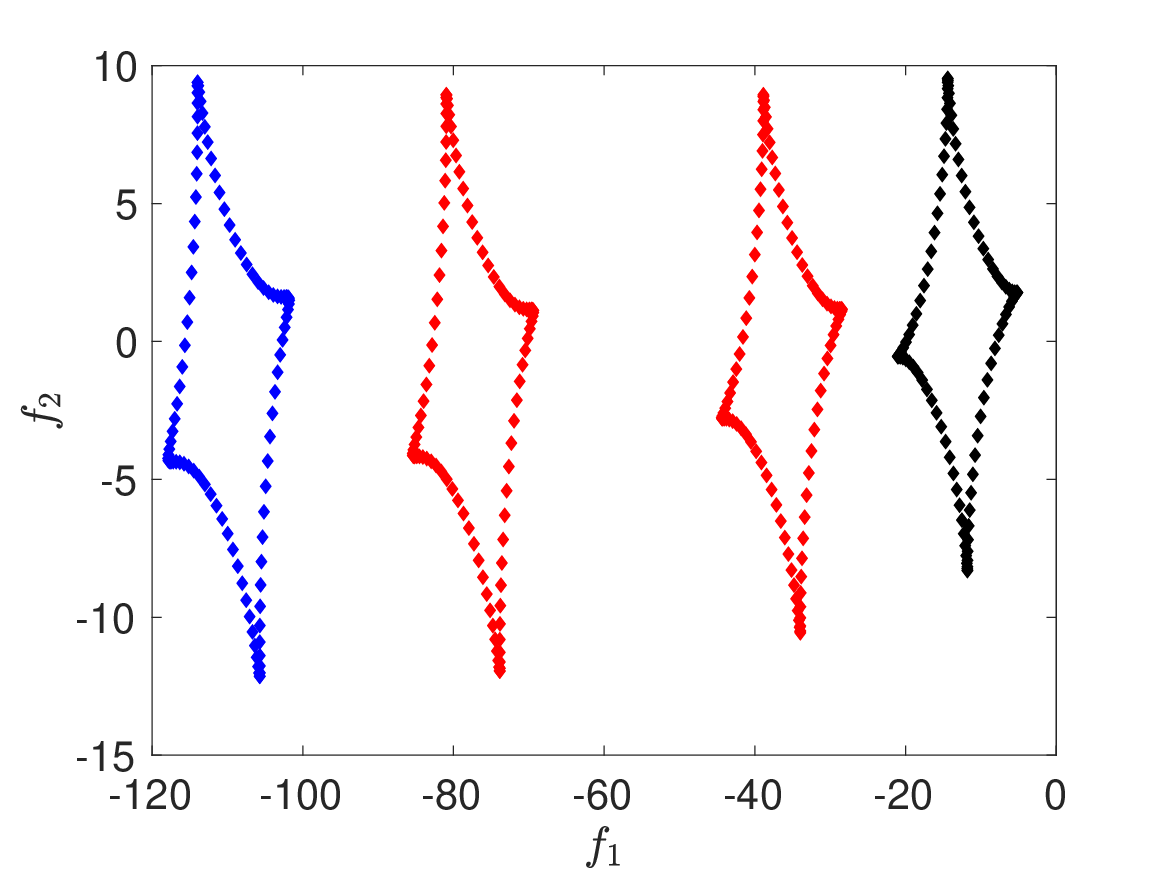}
\label{Iterative_points_image_space_53_second789_cone_K3}}
}
\caption{ Iterative points in the image space generated by Algorithm \ref{avds1} (TRM) with initial point $x_{0}=(9,8)$ for Example  \ref{test_b_inst_53}}
\label{Figure_Iter_Points_Image_Space_Example_53}
\end{figure}

\end{example}

Next, we present a comprehensive comparison of TRM with SD and CGM using the Dolan and Mor{\'e} \cite{dolan2002benchmarking} performance profiles. 

\subsection{Performance profiles}
To generate performance profiles, we run the algorithms for the same set of $100$ initial points for each problem for a maximum allowed $100$ iterations. For any initial point, if the algorithm $s$ converges in less than $100$ iterations, it is called to generate a convergent sequence; otherwise, it is called a nonconvergent sequence. Based on this, we consider a scenario in our analysis, where we look at only those common initial points for which all three algorithms have converged. The problem-wise values of the various metrics obtained from the experiments for each algorithm are provided in Table \ref{uipuy_ytf}. 

In Table  \ref{uipuy_ytf}, we exhibit (a) number of nonconvergences, (b) iterations: average number of iterations taken by an algorithm (averaged over only convergent points), (c) CPU time: time taken (in seconds) by an algorithm (averaged over only convergent points), (d) step size: average step size taken by algorithms (averaged over only convergent points). 
For each metric, the best-performing algorithm's value is marked in bold. Each problem is identified with a name, dimension $n$ in argument space, and dimension $m$ in image space. $\uparrow$ means higher values are better, and $\downarrow$ means lower values are better. A dash $`-$' means that we did not get even a single initial point for which all three algorithms converged.

\begin{table}[!h]
\label{uipuy_ytf}
\caption{Performance of SD, CGM  and TRM methods.}
\centering
 \resizebox{\textwidth}{!}{\begin{minipage}{\textwidth}
  \begin{tabular}{c c|c|c c c|c c c|c c c|c c c}
    \hline
   \multicolumn{3}{c|}{\textbf{Type of SOP}}   &
      \multicolumn{3}{c|}{\textbf{No. of nonconvergences $\downarrow$}} &
      \multicolumn{3}{c|}{\textbf{CPU Time} (Convergent) $\downarrow$} &
      \multicolumn{3}{c|}{\textbf{Iterations} (Convergent) $\downarrow$} & \multicolumn{3}{c}{\textbf{Step Size} (Convergent) $\uparrow$}\\
      \hline
 Name& $n,m$ & Domain ($S$) & \textbf{SD} & \textbf{CGM} & \textbf{TRM} & \textbf{SD} & \textbf{CGM} & \textbf{TRM} & \textbf{SD} & \textbf{CGM} & \textbf{TRM} &\textbf{SD} & \textbf{CGM} & \textbf{TRM}\\
  \hline 
  $\text{GGTZ1-ZDT1}$  &$2,2$& $ [0,1]^{n}$&$73$ & $100$ & $\bf{15}$ & $-$ & $-$ & $-$ & $-$ & $-$ & $-$ & $-$ & $-$ & $-$  \\
  
 & $5,2 $ &  &$98$ & $100$ & $\bf{2}$ & $-$& $-$ & $-$ & $-$ & $-$ & $-$ & $-$ & $-$ & $-$\\
 
 &$8,2$ & & $100$ &$100$ & $\bf{12}$ & $-$& $-$ & $-$ & $-$ & $-$ & $-$ & $-$ & $-$ & $-$ \\
 
 &$10,2$ &  &$100$ & $100$ & $\bf{17}$  & $-$ & $-$ & $-$& $-$ & $-$ & $-$ & $-$ & $-$ & $-$\\
 \hline
 
 $\text{GGTZ2-ZDT4}$  &$10,2$& $[0.01,1]\times$ &$100$ & $100$ & $\bf{0}$ & $-$ & $-$ & $-$ & $-$ & $-$ & $-$ & $-$ & $-$ & $-$ \\
  && $ [-5,5]^{n-1}$ & & &  &  &  &  &&  &  &&  & \\
 
 \hline
 
 $\text{GGTZ3-DTLZ1}$   & $6, 4$ &$[0,1]^{n}$& $94$ & $100$ & $\bf{67}$ & $-$& $-$& $-$& $-$ & $-$ & $-$ & $-$ & $-$ & $-$ \\
  \hline
  
 $\text{GGTZ4-DTLZ3}$  & $5,4$ & $[0,1]^{n}$ &$100$ & $100$ & $\bf{55}$ & $-$ & $-$ & $-$ & $-$ &$ -$ & $-$ & $-$ & $-$ & $-$ \\
 \hline 
 
 $\text{GGTZ6-DTLZ5}$  & $3,3$ & $[0,1]^{n}$&$27$ & $\bf{17}$ & $49$  & $\bf{4.41}$  & $7.37$ & $65.21$ & ${5.63}$ & $7.24$ & $\bf{5.26}$ &${0.095}$ &$ 0.108$ & $\bf{0.202}$ \\
 
 &$5, 3$ & & $56$ & $60$ & $\bf{30}$ & $\bf{6.35}$ & $21.39$  & $ 169.90$ & $\bf{6.3}$ & $16$ & $7.65$ & $0.159$ & ${0.064}$ &$\bf{0.279}$\\
 
 &$7, 5$ & &$97$ & $100 $ & $\bf{72}$ & $-$ & $-$ & $-$ & $-$ & $-$ & $-$ & $-$ & $-$ & $-$\\
 \hline
 
 
 $\text{GGTZ9-Hil}$  & $2,2$ & $[0,5]^{n}$ & $\bf{0}$ &$35$& $40$& $\bf{31.44}$ & $177.26$ & $214.51$ & $\bf{7.19}$ & $27.83$ & $20.58$ &$0.099$ & $\bf{0.680}$ & ${0.053}$\\
 \hline
 
 $\text{GGTZ7-DGO1}$  & $1,2$ & $[-10,13]$ &$\bf{20}$ & $\bf{20}$ & $93$ & $3.02$ & $\bf{2.93}$ & $33.76 $ & $33.76$ & $\bf{1}$ & $2.25$ & $0.575$ & $\bf{1}$ & ${0.323}$\\
 
 $\text{GGTZ8-DGO2}$  & $ 1,2$ & $[-9,9]$ &$\bf{1}$ &$\bf{1}$ &$2$ &$61.29$ & $\bf{6.13}$ & $ 44.9$  & $49.89$ & $\bf{1.59}$ & $ 6.38$&$0.742$ & $\bf{6.017}$ & $0.742$ \\
 \hline
 
  $\text{GGTZ10-JOS1a}$
   & $5, 2$ & $[-2,2]^{n}$ & $21$ & $\bf{7}$ & $42$ & $313.55$ & $1676.95$ & $\bf{132.16}$ & $14$ & $\bf{2.39}$ & $24.7$ & ${0.17}$ & $\bf{1.87}$ & $0.29$\\
 \hline
 
 $\text{GGTZ5-FDSa}$  & $2,3$ & $[-2,2]^{n}$ &$\bf{21}$ & $46$& $97$ & $10.85$ & $\bf{8.06} $ & $32.72$ & $4$& $3$ & $\bf{2.5}$ & $0.069$ & $\bf{0.103}$ & ${0.059}$ \\
 \hline
 
$\text{GGTZ11-Rosenbrock}$  & $4,3$& $[-2,2]^{n}$ & $100$ & $100$ & $\bf{79}$ & $-$ & $-$ & $-$ & $-$ & $-$ & $-$ & $-$ & $-$ & $-$\\
 \hline
 
 $\text{GGTZ12-Brown and Dennis}$ 
   & $4,5$ & $[-25,25]$ & $39$ & $16$ & $\bf{14}$ & $141.90$ & $\bf{130.31}$ & $267.93$ & $26.83$ & $\bf{9.17}$ & $10.67$ & $\bf{1.77}$ & ${1.55}$ & $1.64$ \\

  &  & $\times [-5,5]^{2}$ &  &  &  &  &  &  & &  &  &  &  &  \\

   &  & $\times [-1,1]$ &  &  &  &  & &  &  &  &  & &  &  \\
 \hline 
 
 $\text{GGTZ13-Trigonometric}$
 & $4, 4$  & $[-1,1]^{n}$ &$13$ & $14$ & $\bf{5}$ & $\bf{39.31}$ & $8249.51$ & $ 84.79$& $8.08$ & $4.85$ & $\bf{3.15}$ & $0.166$ & $\bf{0.514}$ & ${0.144}$\\
 \hline
 
 $\text{GGTZ14-Das and Dennis}$ & $5 ,2$ & $[-20,20]^{n}$&$\bf{0}$ & $\bf{0}$& $5$ & $-$ & $-$ & $-$ & $-$ & $-$ & $-$ & $-$ & $-$ & $-$ \\
 \hline
$\text{GGCZ15-BQT1} $
 & $1,2$ &$[2,10]^{n}$ &$74$ & $74$ & $\bf{73}$& $\bf{0.95}$ & $ 2.04$ & $ 1.30$ & $13.15$ & $\bf{1.54}$ & $2.80$& $0.11$ & $\bf{192608174.14}$ &$
         0.21$\\
\hline
$\text{GGCZ16-BQT2}$  & $2,2$ &$[-20,20]^{n}$ & $55$ & $56$ & $\bf{52}$ & $\bf{31.78}$ & $  47.65$ &  $36.58$ & $5.75$ & $6.25$ & $\bf{3}$ &  $0.03$ & $0.3$ & $\bf{0.55}$\\
\hline
$\text{GGTZ17-Sphere}$ & $3,3$ & $[0,1]^{3} $ &$23$ & $\bf{14}$ & $66$ & $98.97$ & $542.82$ & $\bf{87.16}$ & $\bf{5.08}$ & $9.78$ & $10.35$& $0.123$ & ${0.067}$ & $\bf{0.193}$\\
 \thickhline
 \end{tabular}
 \end{minipage}}
\end{table}

First, we look at the convergence property of each algorithm. For this, Figure \ref{per_prof_no_of_failure} shows the performance profile of the three algorithms in terms of the number of nonconvergence metrics. We observe that, overall, TRM has better convergence than SD and CGM for almost all values of $\tau$. 

Next, in Figure \ref{tyu_87}, we report the speed of convergence of the three algorithms with respect to four metrics: number of nonconvergences, number of iterations, CPU time, and average reciprocal step size. For this, we consider only the initial points for which all three algorithms converged. This enables us to do a fair comparison of the speed of convergence by removing the effect of any algorithm failing to converge. In doing so, we see that for $10$ out of the $22$ problems, CGM and SD did not converge for any initial point out of 100. Therefore, we consider only the remaining $12$ problems for this particular analysis.

From Figure \ref{Iterations_Common_conv}, we observe that CGM converges the fastest and performs better than TRM and SD. Between SD and TRM, TRM converges faster by taking fewer iterations.  

In terms of computation time, from Figure \ref{CPU_time_comm_conv_point}, we see that SD requires the least amount of time whereas TRM consumes the most. This is mainly because TRM involves Hessian computation. SD and CGM, on the other hand, do not contain any such calculation and hence are comparatively faster.

In Figure \ref{Step_sixe_avg}, we show the performance profile in terms of the average step size for which higher values mean better performance. According to Dolan and Mor{\'e} \cite{dolan2002benchmarking}, since the metric $t_{p,s}$ should have lower values for better performance, we use $\frac{1}{\text{step-size}}$ as the performance metric. The step size at $k$-th iteration is calculated by $\lVert x_{k+1}-x_{k} \rVert$. From Figure \ref{Step_sixe_avg}, we observe that CGM is better among the three algorithms. 

Based on these three metrics, we do not find any single algorithm that is the best across all the metrics. Rather, depending on the performance measure of interest, one algorithm can be preferred over the other.

Overall, we observe that TRM is advantageous in terms of a higher likelihood of finding a solution. However, there are problems such as GGTZ9-Hil, GGTZ7-DGO1, GGTZ5-FDSa, etc., where CGM and SD outperform TRM in every aspect. In fact, CGM and SD perform better for problems with fewer variables and function components, whereas TRM is more effective for problems with a higher number of variables and function components. Therefore, our proposed TRM algorithm is not a complete replacement for CGM and SD; instead, all algorithms are essential depending on the performance metric of interest. TRM should be considered as a very strong alternative for the large number of variables.

\begin{figure}[!h]
\centering
\mbox{
\subfigure[Nonconvergences]{\includegraphics[width=.24\textwidth]{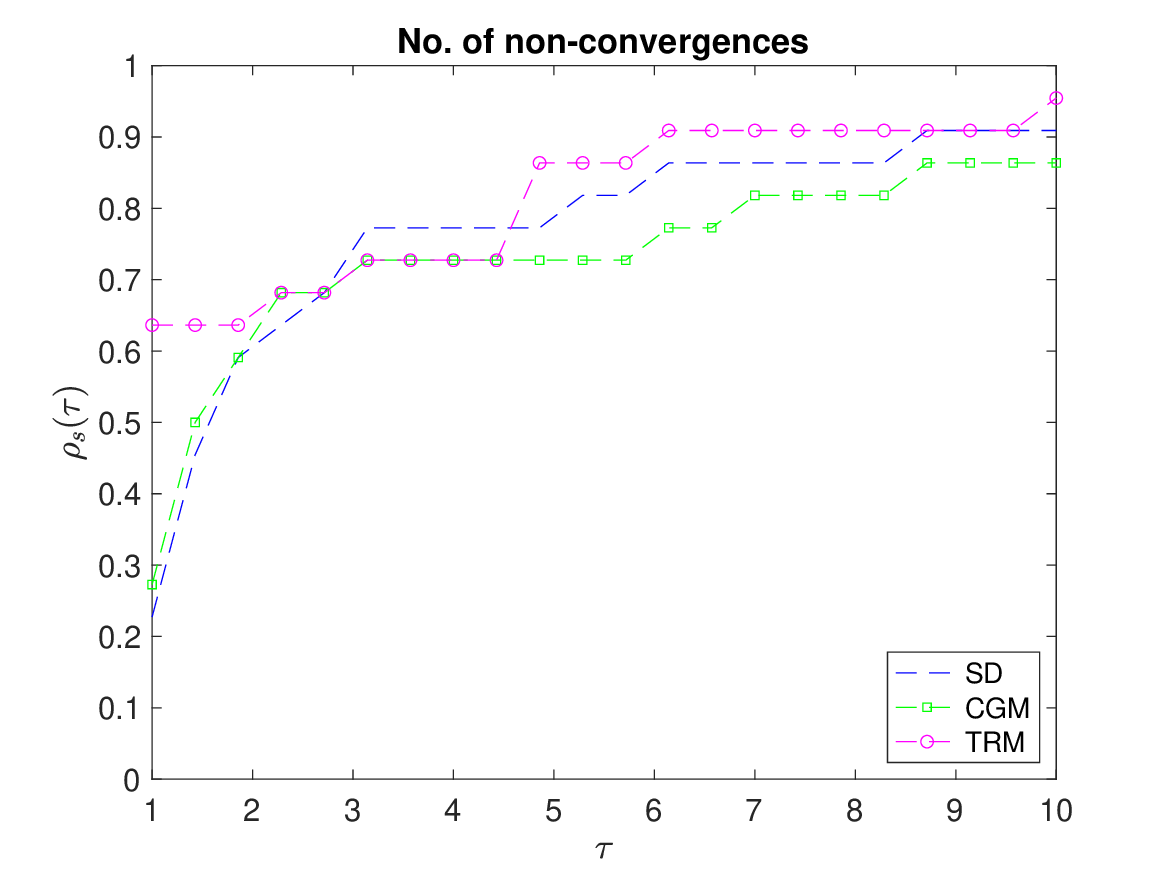}\label{per_prof_no_of_failure}}

\subfigure[Iterations]{\includegraphics[width=.24\textwidth]{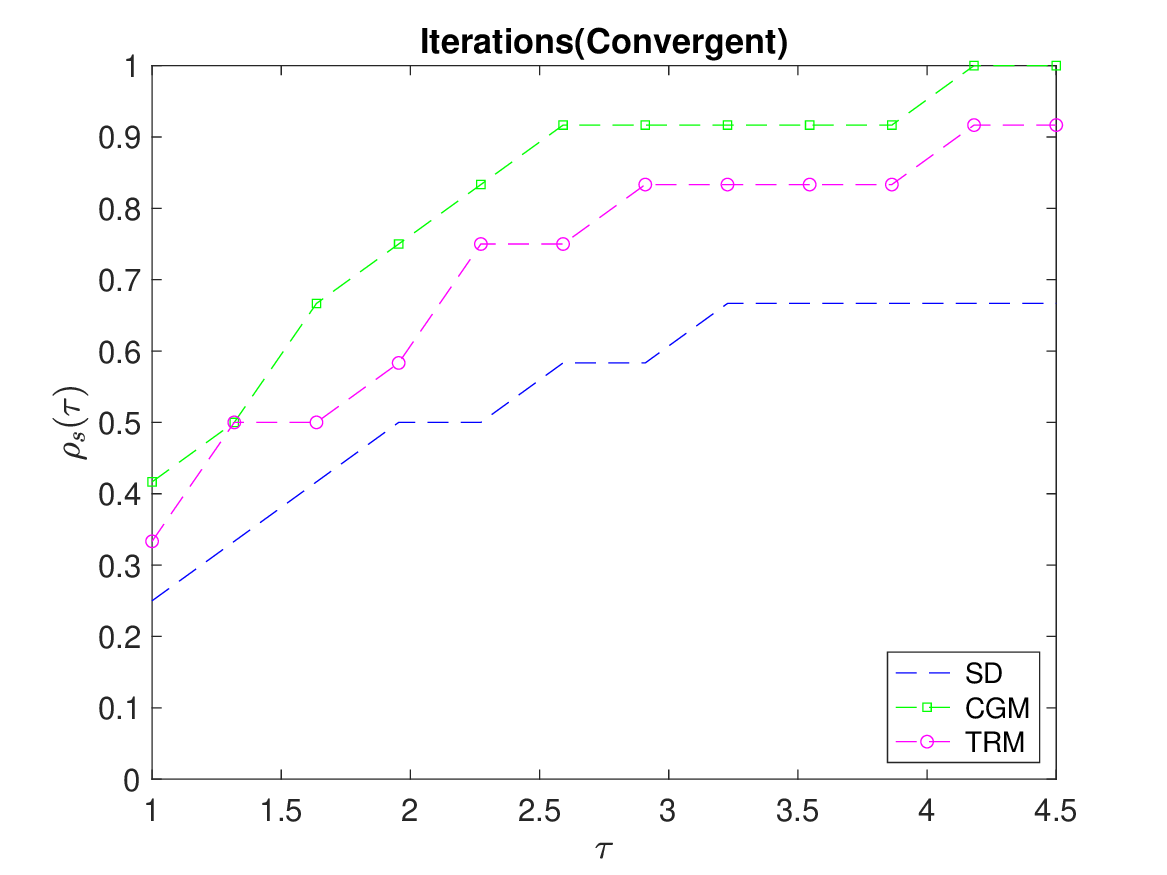}\label{Iterations_Common_conv}}

\subfigure[CPU time (sec)]{\includegraphics[width=.24\textwidth]{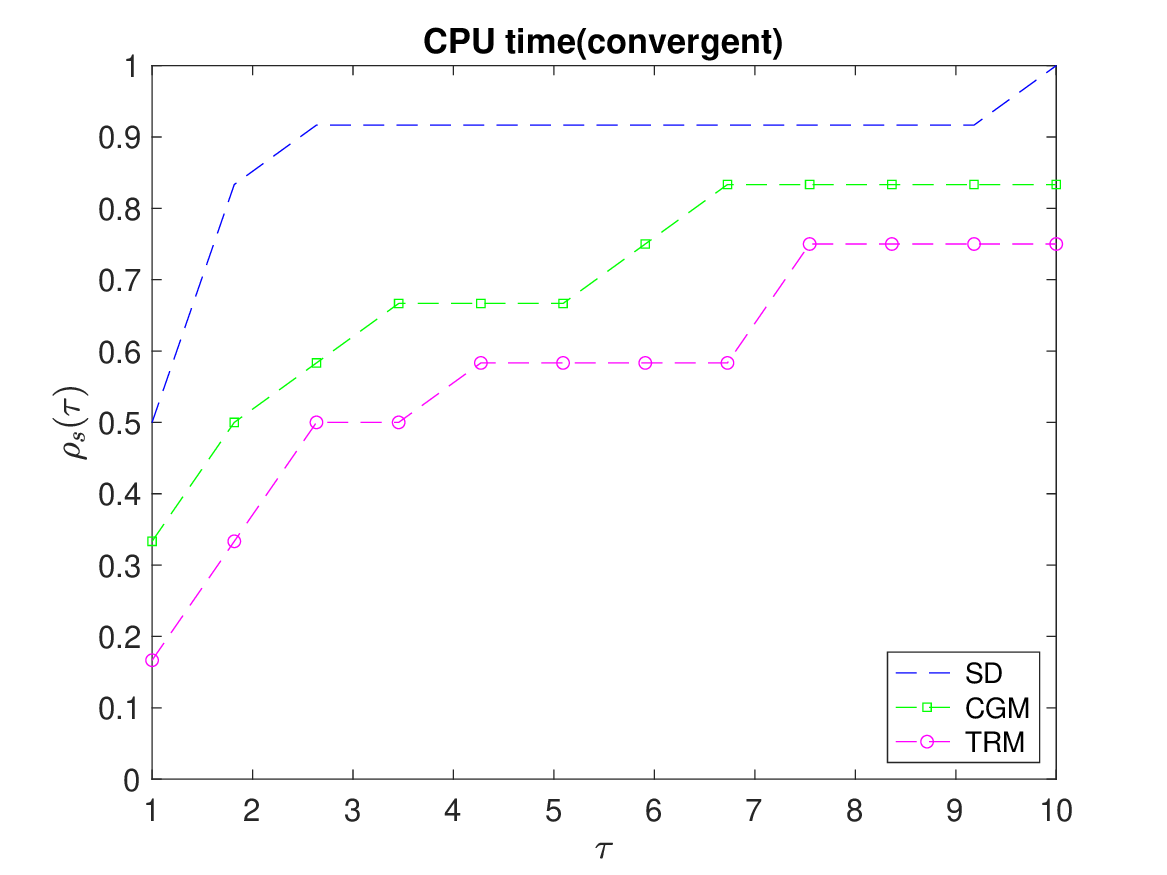}\label{CPU_time_comm_conv_point}}

\subfigure[Step size reciprocal ]{\includegraphics[width=.24\textwidth]{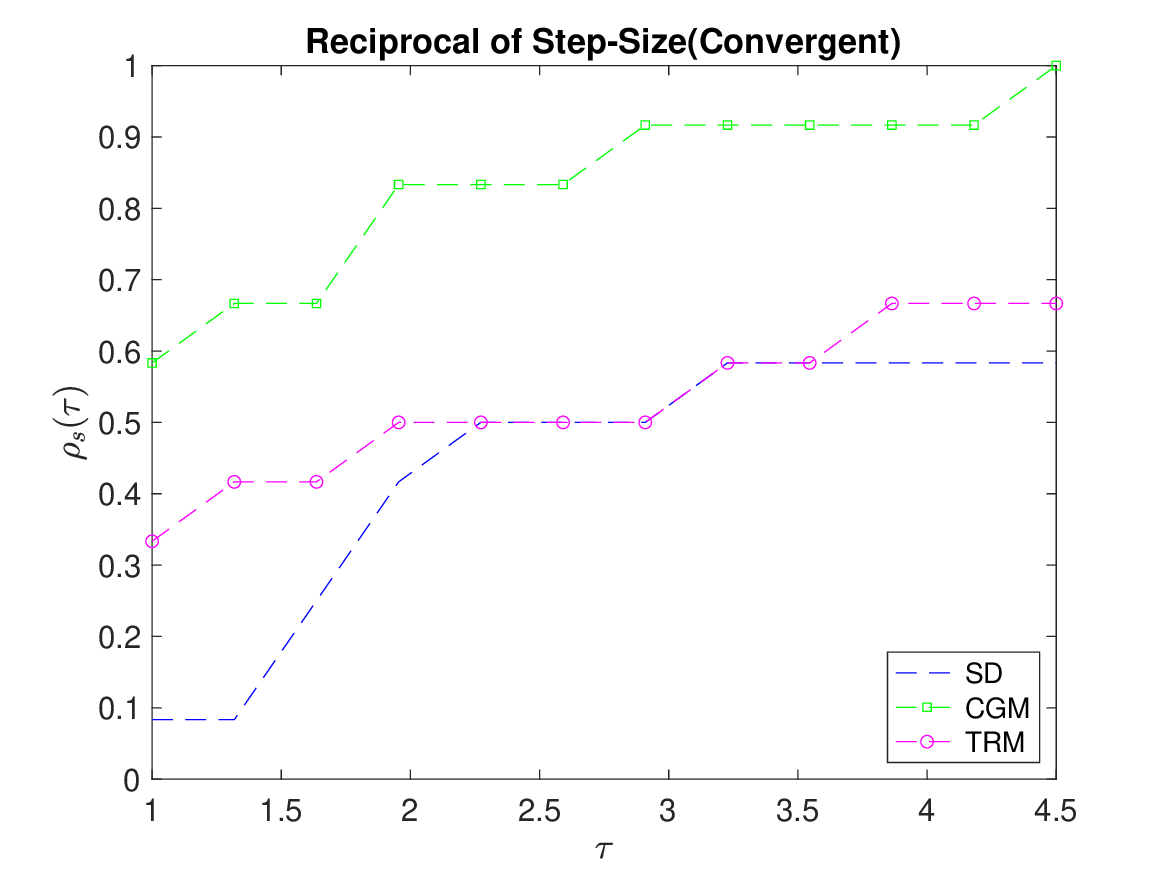}\label{Step_sixe_avg}}
}
\caption{Performance profile of TRM, SD and CGM}
\label{tyu_87}
\end{figure}

\section{Future directions} \label{sect6}
\noindent 
 In the future, we can endeavor to address the following:
\begin{enumerate}[(i)]
\item The proposed trust-region method exhibits a strict monotonic property that requires the objective function to decrease at each iteration. This may limit the speed of convergence. To improve this aspect, we may employ the nonmonotone trust-region method for (\ref{fgcx}), which commonly performs better in scalar and vector optimization. 
 
\item In our convergence analysis, we have used the regularity assumption of \cite{steepmethset}, which can be tried to relax in the future. 

\item In this work, we have considered the objective function of the problem (\ref{fgcx}) as a set-valued map with finite cardinality. In the future, we may explore the applicability of the proposed method for continuum cardinality.  
\end{enumerate}

\appendix 
\section{Test problems for set-valued  optimization}\label{appendix_ind}
Below, we provide a list of test problems for set-valued optimization. These problems are formulated from the test problems of multi-objective optimization problems \cite{huband2006review}. For each of the following test problems, the map $F:\mathbb R^{n} \rightrightarrows \mathbb R^{m}$ is given by $F(x): = \{f^{1}(x), f^{2}(x), \ldots, f^{100}(x)\}$. 
For the DTLZ-type problems, $x_{m}$ is the set of the last $x_{i}$ variables of the vector $x$ in $\mathbb R^{n}$, where $i$ varies from $1$ to $\ell=n-m+1$. The cardinality of the set  $x_{m}$, denoted as $|x_{m}|$, is $\ell$.

\begin{landscape}
\begin{table}
\caption{Test problems for set optimization problems with $p = 100$ and $i \in [p]$}\label{test_prob_table}
\centering 
\resizebox{1.4\textwidth}{!}{\begin{minipage}{\textwidth}
\begin{tabular}{lllll}
\toprule 
\begin{tabular}[c]{@{}l@{}}Problem \\ Name\end{tabular} & $f^i: \mathbb{R}^n \to \mathbb{R}^m$  & Auxiliary Functions                                                   & \begin{tabular}[c]{@{}l@{}}Domain for \\ Initial Points\end{tabular} & $\{(\phi_i, \psi_i): i \in [100]\}$ \\ 
\midrule
\begin{tabular}[c]{@{}l@{}} GGTZ1-ZDT1 \\ $m = 2$ \cite{huband2006review}\end{tabular} & \begin{tabular}[c]{@{}l@{}}$\begin{pmatrix}   f_{1}(x) + \left(0.02+ 0.02 \cos^{16}(\frac{4 \pi i}{100})\right) \cos(\frac{2 \pi i}{100}) \\ g(x) h(f_{1}(x),g(x)) +  0.15+ 0.15\cos^{16}\left(\frac{4 \pi i}{100} \right) \sin(\frac{2 \pi i}{100})  \\ \end{pmatrix}$\end{tabular} & $f_{1}(x) := x_{1}, g(x) := 1 + 9 \sum_{i=2}^{n} x_{i} \text{ and } h(f_{1}, g) := 1 - \sqrt{\frac{f_{1}}{g}}$ & $[0, 1]^n$ & None  \\ \midrule 
\begin{tabular}[c]{@{}l@{}} GGTZ2-ZDT4 \\ $m = 2$ \cite{huband2006review}\end{tabular} & \begin{tabular}[c]{@{}l@{}}$\begin{pmatrix} f_{1}(x)+(1 +  \cos^{16}(\frac{4\pi i}{100}) \cos(\frac{2\pi i}{100})) \\  g(x)h(f_{1}(x),g(x)) + (1 + \cos^{16}(\frac{4\pi i}{100})\sin(\frac{2\pi i}{100}))\\ \end{pmatrix}$\end{tabular} & $\begin{aligned} &f_{1}(x) := x_{1}, g(x) := 1+ 10(n-1)+ \sum_{i=2}^{n} (x^{2}_{i}-10\cos(4\pi x_{i})), \\ 
& h(f_{1},g) := 1 - \sqrt{ f_{1} / g} \end{aligned}$ & $[0, 1]\times[-5,5]^{n - 1}$ & None \\  \midrule 
\begin{tabular}[c]{@{}l@{}} 
GGTZ3-DTLZ1 \\ \cite{huband2006review} \end{tabular} & \begin{tabular}[c]{@{}l@{}}$\begin{pmatrix} (1+g(x))x_{1} x_{2}\cdots x_{m-1} \\ (1+g(x)) x_{1} x_{2}\cdots (1-x_{m-1}) \\ \vdots \\    \frac{1}{2} (1+g(x))\frac{1}{2} x_{1} (1-x_{2})\\ \frac{1}{2}(1-x_{1}) (1+g(x)).\\ \end{pmatrix} + \begin{pmatrix}  \cos(\phi_{i})\sin(\psi_{i})\\ \sin(\phi_{i})\sin(\psi_{i})\\ \left(\cos(\psi_{i})+\ln \tan\frac{\psi_{i}}{2}\right)+0.2 \phi_{i} \\ 0 \\  \vphantom{\int^0}\smash[t]{\vdots} \\ 0 \\ \end{pmatrix}$\end{tabular} & $g(x) := 100 \left(|x_{m}| + \sum_{x_{i}\in x_{m}} (x_{i}-0.5)^{2}-\cos(20 \pi (x_{i}-0.5))\right)$  & $[0, 1]^n$ & 
$\begin{aligned} 
\{\tfrac{\pi}{5}(j - 1): j \in [10]\} \\ 
\times \{\tfrac{\pi}{5}(l - 1): l \in [10]\}
\end{aligned}$ \\    \midrule 
\begin{tabular}[c]{@{}l@{}} 
GGTZ4-DTLZ3 \\\cite{huband2006review}          \end{tabular}                                  & \begin{tabular}[c]{@{}l@{}}$\begin{pmatrix} (1+g(x))\cos(\frac{x_{1}\pi}{2})\cdots \cos(\frac{x_{m-2}\pi}{2}) \cos(\frac{x_{m-1}\pi}{2})\\ (1+g(x))\cos(\frac{x_{1}\pi}{2})\cdots \cos(\frac{x_{m-2}\pi}{2}) \sin(\frac{x_{m-1}\pi}{2})\\ (1+g(x)) \cos(\frac{x_{1}\pi}{2})\cdots \sin(\frac{x_{m-2}\pi}{2}) \\ \vdots \\ (1+ g(x)) \sin(\frac{x_{1}\pi}{2}).\\ \end{pmatrix} + \begin{pmatrix} \sech(\phi_{i}) \cos(\phi_{2})\\ \sech(\phi_{i}) \sin(\psi_{i})\\ \phi_{i}-\tanh(\phi_{i})\\ 0 \\ \vphantom{\int^0}\smash[t]{\vdots} \\ 0 \\ \end{pmatrix}$\end{tabular} & $g(x) := 100 \left(\lvert x_{m} \rvert + \sum_{x_{i} \in x_{m}} (x_{i}-0.5)^{2}-\cos(20\pi(x_{i}-0.5))\right)$ & $[0,1]^n$ & $\begin{aligned} 
\{\tfrac{\pi}{5}(j - 1): j \in [10]\} \\ 
\times \{\tfrac{\pi}{5}(l - 1): l \in [10]\}
\end{aligned}$ \\ \midrule 
\begin{tabular}[c]{@{}l@{}}
GGTZ5-FDSa \\ \cite{NewtonMethod} \end{tabular} & \begin{tabular}[c]{@{}l@{}}$\begin{pmatrix} \frac{1}{n^2} \sum_{i = 1}^{n} i(x_{i}-i)^4 + (1+ \cos(\phi_{i})\cos(\psi_{i})) \\ \exp\left(\sum_{i=1}^{n}\frac{x_{i}}{n}\right) + {\lVert x \rVert}^{2}_{2} + (1 + \cos(\phi_{i})\sin(\psi_{i})) \\ \frac{1}{n(n+1)}  \sum_{i=1}^{n} i(n-i+1) \exp(-x_{i})+ \sin(\phi_{i}) \end{pmatrix}$\end{tabular}  & None   & $[-2,2]^n$ & None   \\ 
\midrule 
\begin{tabular}[c]{@{}l@{}}
GGTZ6-DTLZ5 \\\cite{huband2006review}         \end{tabular}                                    &\begin{tabular}[c]{@{}l@{}}$\begin{pmatrix}  (1+g(x))\cos(\tfrac{\theta_{1}\pi}{2})\cdots\cos(\tfrac{\theta_{m-2}\pi}{2})\cos(\tfrac{\theta_{m-1}\pi}{2})\\  (1+g(x))\cos(\tfrac{\theta_{1}\pi}{2})\cdots\cos(\tfrac{\theta_{m-2}\pi}{2})\sin(\tfrac{\theta_{m-1}\pi}{2})\\  (1+g(x)) \cos(\tfrac{\theta_{1}\pi}{2})\cdots \sin(\tfrac{\theta_{m-2}\pi}{2})\\  \vdots \\ (1+ g(x))\sin(\tfrac{\theta_{1}\pi}{2})\\ \end{pmatrix}+ \begin{pmatrix} 5\tfrac{\psi_{i}}{2\pi}\\  \lambda_{i} \tfrac{\cos(\phi_{i})}{10} \\  \lambda_{i} \tfrac{\sin(\phi_{i})}{10} \\  0 \\  \vphantom{\int^0}\smash[t]{\vdots}\\  0\\ \end{pmatrix}$\end{tabular} & \begin{tabular}[c]{@{}l@{}}$\begin{aligned} & \theta_{i}(x):= \frac{1}{2(1+g(x))} (1+g(x)x_{i}) , i = 2, 3,\ldots,(m-1), \\ &\lambda_{i}=\pi+(\psi_{i}-\pi)^2\text{and}~ g(x) := \sum_{x_{i} \in x_{M}} (x_{i}-0.5)^2. \\ \end{aligned}$\end{tabular} & $[0,1]^n$  & $\begin{aligned} 
\{\tfrac{\pi}{5}(j - 1): j \in [10]\} \\ 
\times \{\tfrac{\pi}{5}(l - 1): l \in [10]\} \end{aligned}$  \\ \midrule 
\begin{tabular}[c]{@{}l@{}}GGTZ7-DGO1\\ $n = 1$~\cite{huband2006review}\end{tabular}   & \begin{tabular}[c]{@{}l@{}}$ \begin{pmatrix} 
\sin x + \sin(\tfrac{\pi i}{50} + \cos(\tfrac{\pi i}{50}))\\    \sin(x + 0.7) + \cos(\tfrac{\pi i}{50} + \sin(\tfrac{\pi i}{50}))\\    \end{pmatrix}$\end{tabular} & None & $[-4\pi, 4\pi]$   & None  \\ \midrule 
\begin{tabular}[c]{@{}l@{}}GGTZ8-DGO2\\ $n = 1$~\cite{huband2006review}\end{tabular}   & \begin{tabular}[c]{@{}l@{}}$\begin{pmatrix} x^{2} + \sin(\tfrac{\pi i}{50} + \cos(\tfrac{\pi i}{50}))\\ 9 -\sqrt{81-x^2} + \cos(\tfrac{\pi i}{50} + \sin(\tfrac{2\pi i}{50}))\\ \end{pmatrix}$\end{tabular}   & None   & $[-3\pi, 3\pi]$   & None  \\ \midrule 
\begin{tabular}[c]{@{}l@{}}GGTZ9-Hil\\ $m = 2$~\cite{NewtonMethod}\end{tabular}   & \begin{tabular}[c]{@{}l@{}}$g(x) + \begin{pmatrix} 10 ((9+ \exp(\sin(\tfrac{\pi i}{25}))- \sin(\tfrac{\pi i}{25})+ 2(\cos(\tfrac{2\pi i}{25}))^2)/128) \cos( \tfrac{\pi i}{50})\\ 10 ((9+ \exp(\sin(\tfrac{ \pi i}{25}))-\sin(\tfrac{\pi i}{25})+ 2(\cos(\tfrac{2\pi i}{25}))^2)/128) \sin(\tfrac{\pi i}{50})\\     \end{pmatrix}$\end{tabular}         & \begin{tabular}[c]{@{}l@{}}$g(x) := \begin{pmatrix} \cos((\tfrac{\pi}{180})(45+40\sin(2\pi x_{1})+ 25 \sin(2\pi x_{2}))(1+0.5 \cos(2\pi x_{1}) ) \\ \sin((\tfrac{\pi}{180})(45+40\sin(2\pi x_{1})+ 25 \sin(2\pi x_{2}))(1+0.5 \cos(2\pi x_{1}))\\     \end{pmatrix} $\end{tabular} & $[0,1]^n$   & None   \\ \midrule 
\begin{tabular}[c]{@{}l@{}}GGTZ10-JOS1a\\ $m = 2$~\cite{huband2006review}\end{tabular} & \begin{tabular}[c]{@{}l@{}}$ \begin{pmatrix} \frac{1}{n} \sum_{i=1}^{n} x^{2}_{i}+ 0.1 \cos(\tfrac{\pi i}{50})\\  \frac{1}{n} \sum_{i=1}^{n} (x_{i}-2)^250 \sin(\tfrac{\pi i}{50})\\ \end{pmatrix}$\end{tabular} & None & $[-100, 100]^n$ & None  \\ \midrule 
\begin{tabular}[c]{@{}l@{}}GGTZ11-\\ Rosenbrock\\ $m = 3$ \cite{mita2019nonmonotone}\end{tabular}             & \begin{tabular}[c]{@{}l@{}}$\begin{pmatrix} 100 (x_{2}-x^{2}_{1})^{2}+ (x_{2}-1)^{2}+ (r^{2}(\cos({\phi_{i}}) \cos(\psi_{i}) \sin(\psi_{i}))) \\           100 (x_{3}-x^{2}_{2})^{2} + (x_{3}-1)^{2} + (r^{2}(\cos({\phi_{i}}) \sin(\psi_{i}) \sin(\psi_{i})))\\           100 (x_{4}-x^{2}_{3})^{2} + (x_{4}-1)^{2}+ (r^{2}(\cos({\phi_{i}}) \sin(\psi_{i}) \cos^{2}(\psi_{i})))\\     \end{pmatrix}$\end{tabular}   & $r = 16$       & $[-2,2]^n$ & $\begin{aligned} 
\{\tfrac{\pi}{5}(j - 1): j \in [10]\} \\ 
\times \{\tfrac{\pi}{5}(k - 1): k \in [10]\}
\end{aligned}$  \\  \midrule 
\begin{tabular}[c]{@{}l@{}}GGTZ12-\\ Brown and Dennis\\ $n =3, m = 4$ \cite{mita2019nonmonotone}\end{tabular} & \begin{tabular}[c]{@{}l@{}}$\begin{pmatrix}    (x_{1}+ \frac{1}{5}x_{2}-\exp({\frac{1}{5}}))^{2}+ (x_{3}+x_{4}\sin(\frac{1}{5})\hspace{-0.1cm}-\cos(\frac{1}{5}))^{2}\hspace{-0.1cm}+\hspace{-0.1cm} \cos(\phi_{i}) \sin(\psi_{i})\\ (x_{1} \hspace{-0.1cm}+ \frac{2}{5}x_{2}-\exp({\frac{2}{5}}))^{2}\hspace{-0.1cm}+\hspace{-0.1cm} (x_{3}+ x_{4}\sin(\frac{2}{5})\hspace{-0.1cm}-\hspace{-0.1cm}\cos(\frac{2}{5}))^{2}\hspace{-0.1cm}+ \hspace{-0.1cm}\sin(\phi_{i}) \sin(\psi_{i})\\     (x_{1}\hspace{-0.1cm}+\hspace{-0.1cm} \frac{3}{5}x_{3}\hspace{-0.1cm}-\hspace{-0.1cm}\exp(\frac{3}{5}))^{2}\hspace{-0.1cm}+\hspace{-0.1cm} (x_{3}\hspace{-0.1cm}+ \hspace{-0.1cm}x_{4}\sin(\frac{3}{5})\hspace{-0.1cm}-\hspace{-0.1cm}\cos(\frac{3}{5}))^{2}\hspace{-0.1cm}+\hspace{-0.1cm} (\cos(\psi_{i})\hspace{-0.1cm}+ \hspace{-0.1cm}\log(\tan(\frac{\psi_{i}}{2})))\hspace{-0.1cm}+\hspace{-0.1cm}0.5\phi_{i} \hspace{-20cm}\\     \end{pmatrix}$\end{tabular} & None    &  $\begin{aligned} 
& [25, 25] \times [-5, 5]^2 \\ 
& [-1, 1]
\end{aligned}$   & 
$\begin{aligned} 
&\{ \tfrac{2\pi}{5}(j-1):j\in [10]\} \\ 
&\times \{ 0.01+ 0.098(k-1):k\in [10]\}
\end{aligned}$ \\  \midrule 
\begin{tabular}[c]{@{}l@{}}GGTZ13-\\ Trigonometric\\ $m = 4$ \cite{mita2019nonmonotone}\end{tabular}          & \begin{tabular}[c]{@{}l@{}}$\begin{pmatrix}  (1 - \cos x_{1} + (1-\cos x_{1}) -\sin x_{1})^{2} + \cos(\phi_{i}) \sin(\psi_{i})  \\   (2 - \cos (x_{1}+x_{2}) + 2 (1- \cos x_{2})-\sin x_{2})^{2} + \sin(\phi_{i})\sin(\psi_{i}) \\   (3\hspace{-0.1cm}-\hspace{-0.1cm}\cos(x_{1}\hspace{-0.1cm}+\hspace{-0.1cm}x_{2}\hspace{-0.1cm}+\hspace{-0.1cm}x_{3})\hspace{-0.1cm}+3(1\hspace{-0.1cm}-\hspace{-0.1cm}\cos x_{3})\hspace{-0.1cm}-\hspace{-0.1cm}\sin x_{3})^2\hspace{-0.1cm}+\hspace{-0.1cm}(\cos(\psi)\hspace{-0.1cm}+\hspace{-0.1cm}\log(\tan(\frac{\psi_{i}}{2})))\hspace{-0.1cm}+\hspace{-0.1cm}0.2 \phi_{i} \\ (4 -\cos(x_{1}+x_{2}+x_{3}+x_{4})+4(1-\cos x_{4})-\sin x_{4})\hspace{-20cm}\\  \end{pmatrix}$\end{tabular}  & None  & $[-1, 1]^n$
& 
$\begin{aligned} 
&\left\{ \tfrac{2\pi}{5}(j-1):j\in [10]\right\} \\ 
&\times \{ 0.01+ 0.098(k-1):k\in [10]\}
\end{aligned}$ \\  \midrule 
\begin{tabular}[c]{@{}l@{}}GGTZ14-\\ Das and Dennis \\ $m = 2$~\cite{mita2019nonmonotone}\end{tabular} & \begin{tabular}[c]{@{}l@{}}$\begin{pmatrix}  (x^{2}_{1}+x^{2}_{2}+x^{2}_{3}+x^{2}_{4} + x^{2}_{5} + (\sin(\frac{i\pi}{50})+ \cos(\frac{i\pi}{50})) \\ (3 x_{1} + 2 x_{2} - \frac{x_{3}}{3} + 0.01(x_{4} - x_{5})^{3} + (\sin(\frac{i\pi}{50})+\cos(\frac{i\pi}{50}))\\ \end{pmatrix}$\end{tabular} & None & $[-20, 20]^n $ & None  \\ \midrule 
\begin{tabular}[c]{@{}l@{}}GGCZ15-BQT1\\ $m = 2~\cite{steepmethset}$\end{tabular}              & \begin{tabular}[c]{@{}l@{}}
$\begin{pmatrix}  x \\ \frac{x}{2} \sin(x)\\ \end{pmatrix} $  
$+ {\cos}^{2}(x) \left[\frac{(i-1)}{4} \begin{pmatrix} 1\\ -1\\ \end{pmatrix} + \bigg(1-\frac{i-1}{4} \bigg)\right]$ \end{tabular}     & None  & None & None                  \\ \midrule 
\begin{tabular}[c]{@{}l@{}}GGCZ16-BQT2\\ $m = 2$~\cite{steepmethset}\end{tabular}                                                               & \begin{tabular}[c]{@{}l@{}}$\begin{pmatrix}  e^{\frac{x_1}{2}}\cos x_{2} + x_{1} \cos x_{2} \sin \tfrac{\pi(i-1)}{50} - x_{2} \sin x_{2} \cos^3 \tfrac{\pi(i-1)}{50} \\  e^{\frac{x_{2}}{20}} \sin x_{1} + x_{1} \sin x_{2} \sin^{3} \tfrac{\pi(i-1)}{50} + x_{2} \cos x_{2} \cos \tfrac{\pi(i-1)}{50}\\ \end{pmatrix}$\end{tabular}                                      & None   & None  & None    \\ \midrule 
\begin{tabular}[c]{@{}l@{}}GGTZ17-\\ Sphere\\ $n = 3, m = 3$\end{tabular} & \begin{tabular}[c]{@{}l@{}}$\begin{pmatrix}      (1 + g(x_3)) \cos u(x_1) \cos v(x_1, x_2, x_3) \\ (1 + g(x_3)) \cos u(x_1) \sin v(x_1, x_2, x_3) \\ (1 + g(x_3)) \sin u(x_1) \\     \end{pmatrix} + \tfrac{1}{16}     \begin{pmatrix} \cos \phi_i \\ \cos \psi_i \sin \phi_i \\     \sin \psi_i \sin \phi_i \\ \end{pmatrix}$\end{tabular} & \begin{tabular}[c]{@{}l@{}}$\begin{aligned} &g(x_3) := (x_3 - \tfrac{1}{2})^2,~ u(x_1) := \frac{\pi x_1}{2}, \\ &v(x_1, x_2, x_3) := \tfrac{\pi (1 + 2 g(x_3) x_2)}{4 \left(1 + g\left(\sqrt{x_1^2 + x_2^2 + x_3^2}\right)\right)}\end{aligned}$\end{tabular} & None  & 
$\begin{aligned} 
\{\tfrac{\pi}{10}(j - 1): j \in [10]\} \\ 
\times \{\tfrac{\pi}{5}(l - 1): l \in [10]\}
\end{aligned}$  \\ 
\bottomrule                        
\end{tabular}
\end{minipage}}
\end{table}
\end{landscape}

\subsubsection*{Funding}
Core Research Grant (CRG/2022/001347) from SERB, India

\subsubsection*{Data availability}
There is no data associated with this paper.

\subsection*{Declarations}    

\subsubsection*{Funding and/or Conflicts of interests/Competing interests} 
The authors do not have any conflicts of interest to declare. They also do not have any funding conflicts to declare.

\subsubsection*{Ethics approval} 
This article does not involve any human and/or animal studies.


\end{document}